\documentclass[12pt]{amsart}
\usepackage{fullpage}
\usepackage[all]{xypic}
\usepackage[hmargin=2.5cm,vmargin=2.5cm]{geometry}
\usepackage{amssymb}
\usepackage{amsthm}
\usepackage{graphicx}
\newtheorem{theorem}{Theorem}[section]
\newtheorem{corollary}[theorem]{Corollary}

\newtheorem{lemma}[theorem]{Lemma}

\newtheorem{remark}[theorem]{Remark}
\newenvironment{myenumerate}{
  \begin{enumerate}
  \setlength{\itemsep}{1pt}
  \setlength{\parskip}{0pt}
  \setlength{\parsep}{0pt}
}{
  \end{enumerate}
}

\begin{document}
\title{ON THE AUTOMORPHISM GROUP OF A SMOOTH SCHUBERT VARIETY} 

\author{S. Senthamarai Kannan}

\address{Chennai Mathematical Institute, Plot H1, SIPCOT IT Park,
Siruseri, Kelambakkam  603103, Tamilnadu, India.}
\email{kannan@cmi.ac.in.}
\maketitle

\begin{abstract}
Let $G$ be a simple algebraic group of adjoint type over the field $\mathbb{C}$ 
of complex numbers. Let $B$ be a Borel subgroup of $G$ containing a maximal 
torus $T$ of $G$. Let $w$ be an element of the Weyl group $W$ and 
let $X(w)$ be the Schubert variety in $G/B$ corresponding to $w$. 
Let $\alpha_{0}$ denote the highest root of $G$ with respect to $T$ and 
$B.$ Let $P$ be the stabiliser of $X(w)$ in $G.$
In this paper, we prove that if $G$ is simply laced and $X(w)$ is smooth, 
then the connected component of the automorphism group of $X(w)$ containing the identity 
automorphism equals $P$ if and only if $w^{-1}(\alpha_{0})$ is a negative root ( see 
Theorem 4.2 ). We prove a partial result in the non simply laced case 
( see Theorem 6.6 ). 
\end{abstract}

\section{Introduction}

Recall that if $X$ is a smooth projective variety over $\mathbb{C}$,  the connected component
of the group of all automorphisms of $X$ containing the identity automorphism is an algebraic group ( see [15, Theorem 3.7, p.17] and [7, p.268] ), which deals also with the case when $X$ may be singular or
it may be defined over any field.  Further, the Lie algebra  of this automorphism group is isomorphic to the space of all tangent vector fields on $X$, that is the space $H^{0}(X, T_{X})$ of all global sections of the tangent bundle $T_{X}$ of $X$ ( see  [15,  Lemma 3.4, p.13] ). 

The aim of this paper is to study the connected component, containing the identity automorphism  of the group of all automorphisms of a smooth Schubert variety $X(w)$ in the flag variety $G/B$ associated to a simple algebraic group $G$ ( see notation in section 2 ).  We give a fairly precise description of the connected component, containing the identity automorphism of the group of all automorphisms of a smooth Schubert variety $X(w)$ in the flag variety $G/B$ associated to a simple simply laced algebraic group $G$ of adjoint type over $\mathbb{C}$, that is of type $A$, $D$, $E$;  in particular, we give a description of a smooth Schubert variety for which this automorphism group  is a subgroup of $G$.

More precisely,  we prove the following results:

Let $w\in W$ be such that $X(w)$ is smooth. Let $A_{w}$ denote 
the connected component of the group of all 
automorphisms of $X(w)$ containing the identity automorphism. Let $T$ be a maximal torus of $G$ contained in $B$ and $B^{+}\supset T$ be the Borel subgroup of $G$ opposite to $B$ determined by $T.$ Let $\alpha_{0}$ denote the highest root of $G$ with 
respect to $T$ and $B^{+}$. 

For the left action of $G$ on $G/B$,  let  $P_{w}$ denote the stabiliser of $X(w)$ in $G.$ Let $\phi_{w}:P_{w} \longrightarrow A_{w}$ be the homomorphism of algebraic groups induced by the action of $P_{w}$ on $X(w).$  For precise notation, see section 2 and section 3. 

\begin{theorem} (see  Theorem 4.2 )

Assume that $G$ is a simple, simply laced algebraic group of adjoint type over $\mathbb{C}$. Then we have 
\begin{enumerate}
\item 
$\phi_{w}:P _{w}\longrightarrow A_{w}$ is surjective. 
\item
$\phi_{w}:P _{w}\longrightarrow A_{w}$ is an isomorphism if and only if  
$w^{-1}(\alpha_{0})$ is a negative root.
\end{enumerate}
\end{theorem}

 \begin{theorem} ( see Theorem 6.6 )

Assume that $G$ is a simple algebraic group of adjoint type which is not simply laced. Then,
$\phi_{w}:P _{w}\longrightarrow A_{w}$ is injective if and only if  
$w^{-1}(\alpha_{0})$ is a negative root.
\end{theorem}

Recall that the vanishing results of the cohomology groups of the restrictions of the homogeneous vector bundles to Schubert varieties have been an important area of research in the theory of algebraic groups ( see [1], [3],  [4], [6], [11], [14], [16] and [18] ).

In this paper, we prove the following vanishing results of the cohomology groups:

By abuse of notation, we denote  by $T_{G/B}$ also the restriction of the tangent bundle of 
$G/B$ to $X(w)$.  \begin{theorem} ( see Theorem 4.1 )

Assume that $G$ is a simple, simply laced algebraic group of adjoint type over $\mathbb{C}$. 
Then we have 
\begin{myenumerate}
\item
$H^{i}(X(w), T_{G/B})=0$ for every 
$i\geq 1$. 
\item
$H^{0}(X(w) , T_{G/B})$ is the adjoint representation 
$\mathfrak{g}$ of  $G$ if and only if  
$w^{-1}(\alpha_{0})$ is a negative root.
\end{myenumerate}
\end{theorem}

\begin{theorem} ( see Theorem 6.5 )

Assume that $G$ is a simple algebraic group of adjoint type over $\mathbb{C}$ which is not simply laced. Then we have
\begin{myenumerate}
\item
$H^{i}(X(w), T_{G/B})=0$ for every 
$i\geq 1$. 
\item
The adjoint representation $\mathfrak{g}$ of $G$ is a $B$-submodule of 
$H^{0}(X(w) , T_{G/B})$  if and 
only if $w^{-1}(\alpha_{0})$ is a negative root.
\end{myenumerate}
\end{theorem}

The results in this paper play an important role in the study of the connected component, 
containing the identity automorphism of the group of all automorphisms of a 
Bott-Samelson-Demazure-Hansen variety ( see [5]  ).
 
The organisation of the paper is as follows:

Section 2 consists of preliminaries from [4], [9], [10] and [11]. 

The strategy of the proof of the results in this paper uses the induction on the  
dimension of Schubert varieties, using their Bott-Samelson-Demazure-Hansen desingularisations and the structure of the indecomposable representations of a Borel subgroup of $SL(2, \mathbb{C})$ (see [2. p.130, Corollary 9.1]). 

In section 3 and section 4,  we assume that $G$ is simply laced. 
We first  describe the $B$-module of the global sections of the homogeneous vector bundles associated to all those $B$-submodules $V$ of the adjoint representation of $G$ which contain the adjoint representation of $B$ ( see Lemma 3.3 ). Using this result, we prove that all higher cohomology groups $H^{i}(X(w), \mathcal{L}(V))$ vanish for any such $B$-module 
$V,$ where $\mathcal{L}(V)$ is the homogeneous vector bundle associated to $V$ ( see Lemma 3.4 ).
In section 4,  we prove the main results in the simply laced case 
applying Lemma 3.4 to a long exact sequence of cohomology groups associated to a certain short exact sequence of $B$-modules. In section 5, we  prove some vanishing results in the non simply laced case. These results are similar to those in section 3. In section 6, we prove the main results in the non simply laced  case. The proofs of these results are  similar to those of the main results in the simply laced case.

\section{Notation  and Preliminaries }
In this section, we set up some notation and preliminaries.
We refer to [9] and [10] for preliminaries in Lie algebras and algebraic groups.
Let $G$ be a simple algebraic group of adjoint type over $\mathbb{C}$ and $T$ a maximal torus of
 $G$.  Let $W=N_{G}(T)/T$ denote the Weyl group of $G$ with respect to $T$ and we denote 
  the set of roots of $G$ with respect to $T$ by $R$. Let $B^{+}$ be a  Borel subgroup of $G$ 
  containing $T$. Let $B$ be the Borel subgroup of $G$ opposite to $B^{+}$ determined by $T$. 
  That is, $B=n_{0}B^{+}n_{0}^{-1}$, where $n_{0}$ is a representative in $N_{G}(T)$ of the longest     element $w_{0}$ of $W$. Let  $R^{+}\subset R$ be 
  the set of positive roots of $G$ with respect to the Borel subgroup $B^{+}$. Note that the set of 
  roots of $B$ is equal to the set $R^{-} :=-R^{+}$ of negative roots.

Let $S = \{\alpha_1,\ldots,\alpha_n\}$ denote the set of simple roots in
$R^{+}.$ For $\beta \in R^{+},$ we also use the notation $\beta > 0$.  
The  reflection in  $W$ corresponding to a root $\alpha$ ( respectively, a simple root $\alpha_i$ ) 
is denoted by $s_{\alpha}$ ( respectively, $s_{i}$ ). Let $\mathfrak{g}$ be the Lie algebra of $G$. 
Let $\mathfrak{h}\subset \mathfrak{g}$ be the Lie algebra of $T$ and  $\mathfrak{b}\subset \mathfrak{g}$ be the Lie algebra of $B$. Let $X(T)$ denote the group of all characters of $T$. 
We have $X(T)\otimes \mathbb{R}=Hom_{\mathbb{R}}(\mathfrak{h}_{\mathbb{R}}, \mathbb{R})$, the dual of the real form of $\mathfrak{h}$. The positive definite 
$W$-invariant form on $Hom_{\mathbb{R}}(\mathfrak{h}_{\mathbb{R}}, \mathbb{R})$ 
induced by the Killing form of $\mathfrak{g}$ is denoted by $(~,~)$. 
We use the notation $\left< ~,~ \right>$ to
denote $\langle \mu, \alpha \rangle  = \frac{2(\mu,
\alpha)}{(\alpha,\alpha)}$,  for every  $\mu\in X(T)\otimes \mathbb{R}$ and $\alpha\in R$.  
Let $\{x_{\alpha}, ~ h_{\beta}: \alpha \in R,  ~  \beta \in S\}$ denote the Chevalley basis of $\mathfrak{g}$
corresponding to $R.$ For $\alpha\in R$, we denote by 
$\mathfrak{g}_{\alpha}$ the one dimensional root subspace of $\mathfrak{g}$ spanned by $x_{\alpha}$. Let $sl_{2,\alpha}$ denote the $3$ dimensional 
Lie subalgebra of $\mathfrak{g}$ generated by $x_{\alpha}$ and $x_{-\alpha}$. 
Let $\leq$ denote the partial order on $X(T)$ given by $\mu\leq \lambda$
if $\lambda-\mu$ is a non negative integral linear combination of simple 
roots. We denote by $X(T)^+$ the set of dominant characters of 
$T$ with respect to $B^{+}$. Let $\rho$ denote the half sum of all 
positive roots of $G$ with respect to $T$ and $B^{+}.$
For any simple root $\alpha$, we denote the fundamental weight
corresponding to $\alpha$  by $\omega_{\alpha}.$ 

For $w \in W,$ let $l(w)$ denote the length of $w$. We define the 
dot action of $W$ on $X(T)\otimes \mathbb{R}$ by $$w\circ\lambda=w(\lambda+\rho)-\rho,  ~  where  ~ w\in W  ~ and ~ \lambda\in X(T)\otimes \mathbb{R}.$$ 
 We set $R^{+}(w):=\{\beta\in R^{+}:w(\beta)\in -R^{+}\}$.  For $w \in W$, 
let $X(w):=\overline{BwB/B}$ denote the Schubert variety in $G/B$ 
corresponding to $w$.

For a simple root $\alpha$, we denote by $P_{\alpha}$ the minimal parabolic subgroup of $G$  generated by $B$ and $n_{\alpha}$, where $n_{\alpha}$ is a representative of $s_{\alpha}$ in $N_{G}(T)$. Let $L_{\alpha}$ denote the Levi subgroup of $P_{\alpha}$
containing $T$. Note that $L_{\alpha}$ is the product of $T$ and the homomorphic image 
$G_{\alpha}$ of $SL(2, \mathbb{C})$ via a homomorphism $\psi:SL(2, \mathbb{C})\longrightarrow L_{\alpha}$ ( see  [11, II, 1.3] ). We denote the intersection of $L_{\alpha}$ and $B$ by $B_{\alpha}$.  
We note that the morphism $L_{\alpha}/B_{\alpha}\hookrightarrow P_{\alpha}/B$  induced by the inclusion $L_{\alpha}\hookrightarrow P_{\alpha}$ is an isomorphism. Therefore,  to compute the cohomology modules $H^{i}(P_{\alpha}/B, \mathcal{L}(V))$ ( $0\leq i \leq 1$ ) for any $B$-module 
$V,$ we treat $V$ as a $B_{\alpha}$-module  and we compute 
$H^{i}(L_{\alpha}/B_{\alpha}, \mathcal{L}(V))$. 

Now,  we recall some preliminaries on Bott-Samelson-Demazure-Hansen varieties 
which we call  BSDH-varieties and some application of Leray spectral 
sequence to compute the cohomology of line bundles on 
Schubert varieties.  We refer to [4] and  [11] for preliminaries related to BSDH varieties.
We refer to [19] for spectral sequences.

We recall that the BSDH-variety corresponding to a reduced expression $\underline i=(i_{1}, i_{2}, \cdots , i_{r} )$ of $w=s_{{i_1}}s_{{i_2}}\cdots s_{{i_r}}$  is defined by \[
Z(w, \underline i) = \frac {P_{\alpha_{i_{1}}} \times P_{\alpha_{i_{2}}} \times \cdots \times 
P_{\alpha_{i_{r}}}}{B \times \cdots
\times B},
\] where the action of $B \times \cdots \times B$ on $P_{\alpha_{i_{1}}} \times P_{\alpha_{i_{2}}}
\times \cdots \times P_{\alpha_{i_{r}}}$ is given by $(p_1, \ldots , p_r)(b_1, \ldots
, b_r) = (p_1 \cdot b_1, b_1^{-1} \cdot p_2 \cdot b_2, \ldots
,b^{-1}_{r-1} \cdot p_r \cdot b_r)$, $ p_j \in P_{\alpha_{i_{j}}}$, $b_j \in B$ and 
$\underline i=(i_1, i_2, \ldots, i_r)$ (see [4, p.64, Definition 2.2.1] ).

We note that for each reduced expression $\underline i$ of $w$, $Z(w, \underline i)$ is a 
smooth projective 
variety. We denote both the natural birational surjective morphism
from $ Z(w, \underline i)$ to $X(w)$ and the composition map $Z(w, \underline i)
\longrightarrow X(w)\hookrightarrow G/B$ by $\phi_w$. We denote the restriction of the homogeneous vector bundle $\mathcal{L}(V)$ to $X(w)$ by  ${\mathcal L}(w, V)$  and its pull back to 
$Z(w, \underline i)$ via the birational morphism $\phi_{w}$ by ${\mathcal L}(w, \underline i, V)$.  

Let $f_r : Z(w, \underline i) \longrightarrow Z(ws_{i_r},
\underline i')$ denote the map induced by the
projection $$P_{\alpha_{i_1}} \times P_{\alpha_{i_2}} \times \cdots \times 
P_{\alpha_{i_r}} \longrightarrow P_{\alpha_{i_1}} \times P_{\alpha_{i_2}}
\times \cdots \times P_{\alpha_{i_{r-1}}},$$ where $\underline i'=(i_1,i_2,\ldots, i_{r-1})$.
 We note that $f_r$ is a 
$P_{\alpha_{i_r}}/B \simeq {\mathbb P}^{1}$-fibration.

Then for $j\geq 0$, we have the following isomorphism of
$B$-linearized sheaves ( see  [11, II, p.366, section 14.1, equation (4)]  and  [8, Theorem 12.11,  p.290] ):

\[
R^{j}{f_{r}}_{*}{\mathcal L}(w, \underline i, V) = {\mathcal L}({ws_{{i_r}}},  \underline i' , 
H^{j}(P_{\alpha_{i_r}}/B, {\mathcal L}(s_{i_r}, V))).
 \hspace{1.5cm}(Iso) \]

We use the following {\em ascending 1-step construction} as a basic tool 
in computing cohomology modules.

Let $\gamma$ be a simple root such that $l(w) = l(s_{\gamma}w) +1$. Let $Z(w, \underline i)$
be a BSDH-variety corresponding to a reduced
expression $w=s_{{i_{1}}}s_{{i_{2}}}\cdots s_{{i_{r}}}$,
where $\alpha_{i_{1}}=\gamma$. Then we have an induced morphism

\[
g: Z(w, \underline i) \longrightarrow P_{\gamma}/B \simeq {\mathbb P}^{1},
\]
with fibres $Z(s_{\gamma}w, \underline i'')$, where $\underline i''=(i_2,i_3, \ldots, i_r)$.
We note that $P_{\gamma}$ acts on both $Z(w, \underline i)$ and on $P_{\gamma}/B$ 
by the left and that the map $g: Z(w, \underline i) \longrightarrow P_{\gamma}/B$ is $P_{\gamma}$-equivariant.

By an application of the Leray spectral sequence together with the fact
that the base is ${\mathbb P}^{1},$ for every $B$-module $V,$ we obtain the 
following exact sequence of $P_{\gamma}$-modules:
$$
0 \to H^{1}(P_{\gamma}/B, R^{j-1}{g}_{*}{\mathcal L}(w, \underline i , V)) 
\to 
H^{j}(Z(w, \underline i) , {\mathcal L}(w, \underline i, V)) \to
H^{0}(P_{\gamma}/B , R^{j}{g}_{*}{\mathcal L}(w, \underline i, V) ) \to 0.
$$

Since for any $B$-module $V,$ the vector bundle ${\mathcal L}(w, \underline i, V)$
on $Z(w, \underline i)$ is the pull back of the homogeneous vector 
bundle ${\mathcal L}(w, V)$ from $X(w)$, we conclude that the cohomology modules
$$H^{j}(Z(w, \underline i) ,~{\mathcal L}(w, \underline i, V))\cong H^{j}(X(w),
~{\mathcal L}(w,V))$$ (see  [4, Theorem 3.3.4 (b)] ), and are independent
 of the choice of the reduced expression $\underline i$. Hence we denote
$H^{j}(Z(w, \underline i) ,~{\mathcal L}(w, \underline i, V))$ by $H^j(w,V)$. 
For a character $\lambda$ of $B$, we denote the one dimensional $B$-module corresponding to $\lambda$ by $\mathbb C_{\lambda}$. Further, we denote the 
cohomology modules $H^{j}(Z(w, \underline i) ,~{\mathcal L}(w,  \underline i, \mathbb C_{\lambda} ))$ by $H^j(w, \lambda)$. 

Recall the following isomorphism of $B$-linearized sheaves ( see  [11, II, p.366,  section 14.1, equation (4) ] and  [8, Theorem 12.11,  p.290] ): \[ R^{j}{g}_{*}{\mathcal L}(w, \underline i, V) = {\mathcal L}(s_{\gamma},   H^{j}(Z(s_{\gamma}w,  \underline i''), {\mathcal L}(s_{\gamma}w, \underline i'', V)))     \hspace{1.5cm}  ( j\geq 0 )  .  \]

We use the simple notation for the cohomology modules and apply the above isomorphism 
in the above short exact sequence to obtain the following short exact sequence of $P_{\gamma}$-modules:
$$
0 \to H^{1}(s_{\gamma}, H^{j-1}(s_{\gamma}w, V)) 
\to 
H^{j}(w, V) \to
H^{0}(s_{\gamma} , H^{j}(s_{\gamma}w, V) ) \to 0.
$$

In this paper, we call the above short exact sequence of $B$-modules 
{\it SES} when ever we use it. 

Let $\alpha$ be a simple root and $\lambda \in X(T)$ be such that 
$\langle \lambda , \alpha \rangle  \geq 0$. Here, we recall 
the following result due to Demazure ( see [6, Page 1] ) on a short exact 
sequence of $B$-modules:

\begin{lemma} Let $K$ denote the kernel of the surjective evaluation
map $H^0({s_\alpha, \lambda}) 
\longrightarrow \mathbb{C}_\lambda$. Then we have
\begin{enumerate}
\item  The sequence \mbox{$0 \longrightarrow K \longrightarrow H^0({s_\alpha, \lambda}) 
\longrightarrow \mathbb{C}_\lambda \longrightarrow 0$} \hspace{0.2cm} of $B$-modules 
is exact.
\item The sequence {$0 \longrightarrow \mathbb{C}_{s_{\alpha}(\lambda)} \longrightarrow 
K \longrightarrow H^0({s_\alpha, \lambda-\alpha})\longrightarrow 0$}\hspace{0.2cm} of $B$-modules  is  exact, whenever
$\langle \lambda , \alpha \rangle \geq 1$. 
\item If  $\langle \lambda,\alpha \rangle=1$ then $ H^0({s_\alpha,\lambda-\alpha})=0$.
\item If $\langle \lambda,\alpha \rangle=0$, then $K=0$ and hence {\rm (2)} does not hold. 
\end{enumerate}
\end{lemma}

We use the following lemma to compute cohomology groups. The following lemma 
is due to Demazure ( see [6, page 1] ). He used this lemma to prove Borel-Weil-Bott's theorem.

\begin{lemma} 
Let $w = \tau s_{\alpha}$, $l(w) = l(\tau)+1$. Then we have 
\begin{enumerate}
\item  If $\langle \lambda , \alpha \rangle \geq 0$ then 
$H^{j}(w , \lambda) = H^{j}(\tau, H^0({s_\alpha, \lambda}) )$ for all $j\geq 0$.
\item  If $\langle \lambda ,\alpha \rangle \geq 0$, then $H^{j}(w , \lambda ) =
H^{j+1}(w , s_{\alpha}\circ \lambda)$ for all $j\geq 0$.
\item If $\langle \lambda , \alpha \rangle  \leq -2$, then $H^{j+1}(w , \lambda ) =
H^{j}(w ,s_{\alpha}\circ \lambda)$ for all $j\geq 0$. 
\item If $\langle \lambda , \alpha \rangle  = -1$, then $H^{j}( w ,\lambda)$ 
vanishes for every $j\geq 0$.
\end{enumerate}
\end{lemma}
\begin{proof}
 Choose a reduced expression for $w=s_{i_1}s_{i_2}\cdots 
s_{i_r}$ with $\alpha_{i_r}=\alpha$. Hence 
$\tau=s_{i_1}s_{i_2}\cdots s_{i_{r-1}}$ is a reduced 
expression for $\tau$.  Let $\underline i=(i_{1}, i_{2}, \cdots  , i_{r} )$  and 
$\underline i'=(i_{1}, i_{2}, \cdots , i_{r-1} )$. Therefore, we have  the morphism $f_r : Z(w, \underline{i}) \longrightarrow Z(\tau, \underline{i}')$ defined as above. Now, the proof of the lemma follows 
from the fact that the functor $H^0(w,-)$ is the
composite of $H^0(\tau, -)$ and $H^0(s_{\alpha}, -)$ ( together 
with the well-known properties of the cohomology groups of line bundles over $\mathbb{P}^1$ ).
For instance, see [8, p.252, III, Ex $8.1$]  and  [11, p.218, Proposition 5.2(b)].
\end{proof}

The following consequence of Lemma 2.2  will be used 
to compute the cohomology modules in this paper.

Let $\pi:\hat{G}\longrightarrow G$ be the simply connected covering of $G$.  
Let  $\hat{L_{\alpha}}$  ( respectively, $\hat{B_{\alpha}}$ )  be the inverse image 
of $L_{\alpha}$  ( respectively, of $B_{\alpha}$ ) in $\hat{G}$. Note that $\hat{L_{\alpha}}/\hat{B_{\alpha}}$ is isomorphic to $L_{\alpha}/B_{\alpha}$. We make use of this isomorphism to use the same notation for the vector bundle on $L_{\alpha}/B_{\alpha}$ associated to a $\hat{B_{\alpha}}$-module. 
\begin{lemma}
 Let $V$ be an irreducible  $\hat{L_{\alpha}}$-module and $\lambda$ be a character of 
 $\hat{B_{\alpha}}$.  Then we have
\begin{enumerate}
\item If $\langle \lambda , \alpha \rangle \geq 0$, then the $\hat{L_{\alpha}}$-module
$H^{0}(L_{\alpha}/B_{\alpha} , V\otimes \mathbb{C}_{\lambda})$ 
is isomorphic to the tensor product of  $ ~ V$ and 
$H^{0}(L_{\alpha}/B_{\alpha} , \mathbb{C}_{\lambda})$. Further, we have  
$H^{j}(L_{\alpha}/B_{\alpha} , V\otimes \mathbb{C}_{\lambda}) =0$ 
for every $j\geq 1$.
\item If $\langle \lambda , \alpha \rangle  \leq -2$, then we have  
$H^{0}(L_{\alpha}/B_{\alpha} , V \otimes \mathbb{C}_{\lambda})=0.$ 
Further, the $\hat{L_{\alpha}}$-module  $H^{1}(L_{\alpha}/B_{\alpha} , V \otimes \mathbb{C}_{\lambda})$ is isomorphic to the tensor product of  $V$ and $H^{0}(L_{\alpha}/B_{\alpha} , 
\mathbb{C}_{s_{\alpha}\circ\lambda})$. 
\item If $\langle \lambda , \alpha \rangle  = -1$, then 
$H^{j}( L_{\alpha}/B_{\alpha} , V \otimes \mathbb{C}_{\lambda}) =0$ 
for every $j\geq 0$.
\end{enumerate}
\end{lemma}

\begin{proof} By [11, p.53, I, Proposition 4.8] and [11, p.77, I, Proposition 5.12], for all $j\geq 0$, we have the following isomorphism as 
$\hat{L_{\alpha}}$-modules:
   $$H^j(L_{\alpha}/B_{\alpha}, V \otimes \mathbb C_{\lambda})\simeq V \otimes
   H^j(L_{\alpha}/B_{\alpha}, \mathbb C_{\lambda}).$$ 
Now, the proof of the lemma follows from Lemma 2.2 by taking $w=s_{\alpha}$ 
and the fact that $L_{\alpha}/B_{\alpha} \simeq P_{\alpha}/B$.
\end{proof}

We now state the following Lemma on indecomposable 
$\hat{B_{\alpha}}$ ( respectively,  $B_{\alpha}$ ) modules which will be used in computing 
the cohomology modules ( see [2, p.130, Corollary 9.1]  ).

\begin{lemma}
\begin{enumerate}
\item
Any finite dimensional indecomposable $\hat{B_{\alpha}}$-module $V$ is isomorphic to 
$V^{\prime}\otimes \mathbb{C}_{\lambda}$ for some irreducible representation
$V^{\prime}$ of $\hat{L_{\alpha}}$, and some character $\lambda$ of $\hat{B_{\alpha}}$.

\item
Any finite dimensional indecomposable $B_{\alpha}$-module $V$ is isomorphic to 
$V^{\prime}\otimes \mathbb{C}_{\lambda}$ for some irreducible representation
$V^{\prime}$ of $\hat{L_{\alpha}}$, and some character $\lambda$ of $\hat{B_{\alpha}}$.
\end{enumerate}
\end{lemma}
\begin{proof}
The proof  of part 1 follows from [2, p.130, Corollary 9.1].

The proof  of part 2 follows from the fact that every $B_{\alpha}$-module can be viewed  as a 
$\hat{B_{\alpha}}$-module via the natural homomorphism. 
\end{proof}

The following lemmas on the evaluation map are known. For the sake of completeness, 
we provide a sketch of a proof.
For a $B$-module $V,$  the evaluation map $ev: H^{0}(w, V)\longrightarrow V$ is given 
by $ev(s)=s(idB)$.

\begin{lemma}
Let $V$ be a finite dimensional rational $G$-module, and let $w\in W$. 
Then we have

\begin{myenumerate}
\item
The evaluation map 
$ev: H^{0}(w, V)\longrightarrow V$ is an isomorphism of $B$-modules.
\item
$H^{i}(w, V)=0$ for every $i\geq 1$.
\end{myenumerate}
\end{lemma}
\begin{proof}

Since $V$ is a $G$-module, the homogeneous vector bundle $\mathcal{L} ( V)$
on $G/B$ is trivial. Therefore, its restriction  $\mathcal{L}(w, V)$ to $X(w)$  is also trivial.
Hence,  the assertions of both parts of the lemma  follow immediately.
\end{proof}  

\begin{lemma}
Let $w\in W,$ and $V$ be a $B$-submodule 
of a $G$-module $V^{\prime}$. Then the evaluation map 
$ev:H^{0}(w, V)\longrightarrow V$ is injective.
\end{lemma}
\begin{proof}

Since $V$ is a $B$-submodule of $V^{\prime}$,  it follows that $H^{0}(w , V)$ is a $B$-submodule of $H^{0}(w , V^{\prime})$. Hence, we have the following commutative diagram of 
$B$-modules: 

\begin{center}
$\xymatrix{H^{0}(w , V)\ar[d]\ar[r] & H^{0}(w, V^{\prime})\ar[d]\\
V \ar[r] &  V^{\prime}}$
\end{center} 
Here, both the horizontal maps are the canonical inclusions and both the 
vertical maps are evaluation maps.
Since the  $B$-module $V^{\prime}$ is the restriction 
of a $G$-module, by using Lemma 2.5, we see that the evaluation map
$ev^{\prime}:H^{0}(w , V^{\prime})\longrightarrow V^{\prime}$ on the 
right hand side is an isomorphism of $B$-modules. Since the first 
horizontal map $H^{0}(w, V)\longrightarrow H^{0}(w, V^{\prime})$ is 
injective, the composition $H^{0}(w, V)\longrightarrow V^{\prime}$ is also 
injective. 
Thus, we conclude that the vertical map $ev: H^{0}(w, V) \longrightarrow V$ 
on the left hand side is injective. This completes the proof of the lemma.
\end{proof}

\section{Cohomology modules in the simply laced case}
In this section, we prove some preliminary results for the simply laced case. 
We use these results in section 4 to prove the main results. 
Through out this section, we assume that  $G$ is simply laced. 

We use the following lemma whose proof is well known.

\begin{lemma}
Let $\alpha\in S$ and $\beta$ be a root different from both 
$\alpha$ and $-\alpha.$
Then we have $\langle \beta , \alpha \rangle \in \{-1, 0, 1\}$.
\end{lemma}

From now on, for any $T$-module and a character $\mu$ of $T$, we denote the set of all vectors 
$v$ in $V$ such that $t  v=\mu(t)v$ for every $t\in T$ by $V_{\mu}$. That is, 
$$V_{\mu}:=\{v\in V: t v=\mu(t)v ~ for ~ all ~ t\in T\}.$$    Here $tv$ denotes the action of $t$ on $v$.

{\it Notation. } We set up a notation for some indecomposable $B_{\gamma}$-summand of 
$\mathfrak{g}.$  Let $\gamma$ be a simple root. We recall that $sl_{2, \gamma}$ is the simple Lie algebra corresponding to $\gamma$. We first note that $sl_{2, \gamma}$ is an indecomposable $B_{\gamma}$-summand of $\mathfrak{g}$. We also note that for $\beta\in R$ such that 
$\langle \beta , \gamma \rangle = 1,$  the $T$-submodule  $\mathfrak{g}_{\beta} \oplus \mathfrak{g}_{\beta-\gamma}$  is an indecomposable $B_{\gamma}$-summand
of $\mathfrak{g}.$ We denote it by $\mathfrak{g}_{\beta, \beta - \gamma}.$ Indeed, this is an irreducible $L_{\gamma}$-submodule of $\mathfrak{g}.$  In this paper, we denote by 
the one dimensional complex vector space generated by a non zero vector $v$ by $\mathbb{C}v$.   

The following lemma gives a description of indecomposable $B_{\gamma}$-summands
of $\mathfrak{g}$.
\begin{lemma}
Every indecomposable $B_{\gamma}$-summand $V$ of 
$\mathfrak{g}$ is one of the following:
\begin{myenumerate}
\item
$V=\mathbb{C} h$ for some non zero vector $h\in \mathfrak{h}$ such that 
$\gamma(h)=0$. 
\item
$V=\mathfrak{g}_{\beta}$ for some root $\beta$ such that 
$\langle \beta , \gamma \rangle = 0.$
\item  $V=\mathfrak{g}_{\beta , \beta-\gamma}$ 
for some root $\beta$ such that 
$\langle \beta , \gamma \rangle = 1$. 
\item
$V=sl_{2,\gamma}$, the three dimensional irreducible 
$L_{\gamma}$-module with highest weight $\gamma$.
\end{myenumerate}
\end{lemma}

\begin{proof}
Let $V$ be an indecomposable $B_{\gamma}$-summand of $\mathfrak{g}$.
Let $\lambda$ be a maximal weight of $V$. Then the direct sum $\oplus_{r\in \mathbb{Z}_{\geq 0}}V_{\lambda-r\gamma}$ is a $B_{\gamma}$-summand of $V$.  
Hence, we have $V=\oplus_{r\in \mathbb{Z}_{\geq 0}}V_{\lambda-r\gamma}$. Note that 
$V$ is a $L_{\gamma}$-module. Therefore, using Lemma 3.1, we see that the dimension of $V$ 
must be at most two unless $V=sl_{2,\gamma}$. 
Further, if the dimension of $V$ is one, either $V=\mathbb{C} h$ for some non zero vector 
$h\in \mathfrak{h}$ such that $\gamma(h)=0,$ or $V=\mathfrak{g}_{\beta}$ for some root $\beta$ such that $\langle \beta , \gamma \rangle = 0.$  Also, if the dimension 
of $V$ is two, then we must have $V=\mathfrak{g}_{\beta , \beta-\gamma}$  for some root $\beta$ such that $\langle \beta , \gamma \rangle = 1$. This completes the proof of the lemma.
\end{proof}

First, we note that a decomposition of a $B_{\gamma}$-module into a direct sum of 
indecomposable submodules is not necessarily unique.  The following lemma is crucial in the proof of the main results in section 3 and section 4.

We make a remark on some notation used in the statement of Lemma 3.3 and its 
proof. 

{\it Notation.} 
Let  $h(\gamma) \in\mathfrak{h}$ be the fundamental dominant coweight corresponding to $\gamma$. That is,  $h(\gamma)$  satisfies $\gamma(h(\gamma))=1$ and $\nu(h(\gamma))=0$ for every 
simple root $\nu$ different from $\gamma$. Then the smallest $B_{\gamma}$-submodule of 
$\mathfrak{g}$ containing $h(\gamma)$ is the two dimensional vector subspace $\mathbb{C} h(\gamma) \oplus \mathfrak{g}_{-\gamma}$. Note that this is an indecomposable $B_{\gamma}$-submodule of $\mathfrak{g}$. In the context of Lemma 2.4,  we have $\mathbb{C} h(\gamma) \oplus \mathfrak{g}_{-\gamma}=V_{1}\otimes \mathbb{C}_{-\omega_{\gamma}}$, where $V_{1}$ is the standard  two dimensional irreducible $\hat{L_{\gamma}}$-module. We denote this indecomposable $B_{\gamma}$-submodule by $\mathfrak{g}_{0, -\gamma}$. Further, $\mathfrak{g}_{0, -\gamma}$ is an indecomposable $B_{\gamma}$-direct summand of any $B_{\gamma}$-submodule $V^{\prime}$ of $\mathfrak{g}$ such that  $\mathfrak{g}_{0, -\gamma}\subset V^{\prime}$  and 
$\mathfrak{g}_{\gamma}\bigcap V^{\prime}=0$.  We use this two dimensional $B_{\gamma}$-submodule in type (2) of the statement of Lemma 3.3 and its proof.

\begin{lemma}
Let $w\in W$ and $\gamma\in S$. Let $V$ be a $B$-submodule of $\mathfrak{g}$ containing  $\mathfrak{b}$. Then there is a decomposition of the $B_{\gamma}$-module $H^{0}(w,V)$ such that 
every indecomposable $B_{\gamma}$-summand $V^{\prime}$ of 
$H^{0}(w , V)$ is one of the following:
\begin{myenumerate}
\item                 
$V^{\prime}=\mathbb{C} h$ for some non zero vector $h\in \mathfrak{h}$ such that 
$\gamma(h)=0$. 
\item  
$V^{\prime}=\mathfrak{g}_{0, -\gamma}$.
\item
$V^{\prime}=\mathfrak{g}_{\beta}$ for some root $\beta$ such that 
$\langle \beta , \gamma \rangle \in \{-1, 0 \}$. 
\item  $V^{\prime}=\mathfrak{g}_{\beta , \beta-\gamma}$ 
for some root $\beta$ such that 
$\langle \beta , \gamma \rangle = 1.$  
\item
$V^{\prime}=sl_{2,\gamma}$, the restriction of the three dimensional irreducible 
$L_{\gamma}$-module with highest weight $\gamma$.
\end{myenumerate}
\end{lemma}
\begin{proof}
First note that in view of Lemma 2.6,  $H^{0}(w, V)$ is a $B$-submodule
of $V$. Let $V^{\prime}$ be an indecomposable $B_{\gamma}$-summand of
$H^{0}(w, V)$.  Note that there is a non zero vector $v\in V^{\prime}$ and a $\mu\in X(B_{\gamma})=X(T)$ such that $bv=\mu(b) v$ for every $b\in B_{\gamma}$.
If  $\mu\neq -\gamma$, then using the arguments similar to the proof of Lemma 3.2, we see that $V^{\prime}$ must be one of the 
types (1), (3) or (4).

Therefore, we may assume that $\mathfrak{g}_{-\gamma}$ is a $B_{\gamma}$-submodule of 
$V^{\prime}$. To complete the proof of the lemma, we need to show that  either  $sl_{2, \gamma}$ or  
$\mathfrak{g}_{0, -\gamma}$ is an indecomposable $B_{\gamma}$-summand of $H^{0}(w, V).$  

The proof  is by induction on $l(w)$.
If $l(w)=0$, then we have $w=1$.  First note that since $S$ is a basis of the complex vector space $Hom_{\mathbb{C}}(\mathfrak{h}, \mathbb{C})$, it follows that  $h(\gamma)\in \mathfrak{h}$. Further,
since $\mathfrak{h}$ is a complex vector subspace of $V,$ it follows that  $h(\gamma)\in V$.
Thus,  type (2) indecomposable $B_{\gamma}$-module $\mathfrak{g}_{0, -\gamma}$ is a submodule of $V$. Now, if $\mathfrak{g}_{\gamma}\subset V,$ then it is easy to see that $V^{\prime}=sl_{2, \gamma}$ is an indecomposable $B_{\gamma}$-summand of $V$ containing $\mathfrak{g}_{-\gamma}$. Otherwise, the $B_{\gamma}$-submodule $\mathfrak{g}_{0, -\gamma}$ is a summand of $V$.

Therefore, we may  choose a simple root 
$\alpha$ such that $l(w)=1+l(s_{\alpha}w)$.
By induction on $l(w)$, we assume that for any simple root $\nu$ such that 
$\mathfrak{g}_{-\nu}\subset H^{0}(s_{\alpha}w, V)$, there is a 
decomposition of the $B_{\nu}$-module $H^{0}(s_{\alpha}w, V)$  such that the 
indecomposable $B_{\nu}$-summand $V^{\prime}$ of $H^{0}(s_{\alpha}w, V)$ 
containing $\mathfrak{g}_{-\nu}$ is either of the form 
$V^{\prime}=\mathfrak{g}_{0, -\nu}$ or of the form
$V^{\prime}=sl_{2, \nu}$. In particular,  it follows that 
either $\mathfrak{g}_{0, -\gamma}$  or $sl_{2, \gamma}$ is a 
 $B_{\gamma}$-summand of  $H^{0}(s_{\alpha}w, V)$.

We divide the  proof into three different cases as follows.

{\it Case 1:} We first assume that $\gamma=\alpha$. 
By using {\it SES}, we have $$H^{0}(s_{\alpha}, H^{0}(s_{\alpha}w , V))=H^{0}(w, V).$$ 
 
 Note that in view of Lemma 2.4,  we have $\mathfrak{g}_{0, -\gamma}=V_{1}\otimes \mathbb{C}_{-\omega_{\gamma}}$, where $V_{1}$ is the standard  two dimensional irreducible 
$\hat{L_{\gamma}}$-module. Hence by Lemma 2.3, we have  $$H^{0}(s_{\alpha}, \mathfrak{g}_{0, -\gamma})=H^{0}(s_{\gamma}, \mathfrak{g}_{0, -\gamma})=0 ~ ( ~ since ~ \alpha=\gamma ~ ).$$  
Therefore, if $\mathfrak{g}_{0, -\gamma}$  is a $B_{\gamma}$-summand of 
$H^{0}(s_{\alpha}w, V)$, then by using Lemma 2.6, we see that  $\mathfrak{g}_{-\gamma}$ can not be a subspace of $H^{0}(s_{\alpha},  H^{0}(s_{\alpha}w, V))=H^{0}(w, V)$. This is a contradiction to the above observation. Hence, we conclude that $sl_{2, \gamma}$ is a $B_{\gamma}$-summand of $H^{0}(s_{\alpha}w,  V)$. Further,  by Lemma 2.5, we  have 
$$H^{0}(L_{\alpha}/B_{\alpha},  sl_{2, \gamma})=H^{0}(L_{\gamma}/B_{\gamma},  sl_{2, \gamma})=sl_{2, \gamma} ~ ( ~ since ~ \alpha=\gamma  ~ ).$$
This completes the proof for the case when $\gamma=\alpha$.

{\it Case 2 :}
We assume that $\alpha$ is different from $\gamma$ and 
$\langle \gamma , \alpha \rangle \neq 0$.
By using Lemma 3.1, we have $\langle \gamma , \alpha \rangle = -1$.

{\it Sub case 1:} 
Assume that $\mathfrak{g}_{0, -\gamma}$  is a $B_{\gamma}$-summand of $H^{0}(s_{\alpha}w, V)$. 
Note that $h(\gamma)\in \mathfrak{h}\bigcap \mathfrak{g}_{0, -\gamma}$. Since $\alpha\neq \gamma$, we have $\alpha(h(\gamma))=0$. Hence,  
$\mathbb{C} h(\gamma) $ is a $B_{\alpha}$-direct summand 
of $H^{0}(s_{\alpha}w, V)$. Hence, $\mathbb{C} h(\gamma) $ must be a 
$B_{\alpha}$-submodule of $H^{0}(s_{\alpha}, H^{0}(s_{\alpha}w, V))=H^{0}(w, V)$.
Since $\mathfrak{g}_{-\gamma}\subset H^{0}(w, V),$ it follows that 
$\mathfrak{g}_{0, -\gamma}$ is a $B_{\gamma}$-summand of $H^{0}(w, V)$.

{\it Sub case 2:}
Assume that $sl_{2, \gamma}$ is a $B_{\gamma}$-summand of $H^{0}(s_{\alpha}w, V)$.
Now, if $\mathfrak{g}_{\alpha +\gamma}$ is a subspace of 
$H^{0}(s_{\alpha}w, V)$, then $\mathfrak{g}_{\alpha+\gamma, \gamma}$ is an 
indecomposable $B_{\alpha}$-summand of $H^{0}(s_{\alpha}w, V)$. Since $\mathfrak{g}_{\alpha+\gamma, \gamma}$ is a $\hat{L_{\alpha}}$-module, using 
Lemma  2.5, we see that  $$H^{0}(s_{\alpha}, \mathfrak{g}_{\alpha+\gamma, \gamma})=\mathfrak{g}_{\alpha+\gamma, \gamma}\subset H^{0}(s_{\alpha},  H^{0}(s_{\alpha}w, V)).$$ 
Thus, we see that the $B_{\gamma}$-span $sl_{2,\gamma}$ of $\mathfrak{g}_{\gamma}$ 
must be a $B_{\gamma}$-summand of $H^{0}(s_{\alpha}, H^{0}(s_{\alpha}w, V))=H^{0}(w, V)$. 

On the other hand, if $\mathfrak{g}_{\alpha +\gamma}$ is not a subspace of 
$H^{0}(s_{\alpha}w, V)$, then $\mathfrak{g}_{\gamma}$ is an indecomposable 
$B_{\alpha}$-direct summand of $H^{0}(s_{\alpha}w, V)$.
Since $\langle \gamma , \alpha \rangle =-1$, by Lemma 2.2 , we have
$H^{i}(s_{\alpha} ,  \mathfrak{\gamma}) =0$ for every $i \in \mathbb{Z}_{\geq 0}$. In particular, $\mathfrak{g}_{\gamma}$ can not be a subspace of $H^{0}(s_{\alpha}, H^{0}(s_{\alpha}w, V))$. Hence, $sl_{2, \gamma}$ is not a $B_{\gamma}$-summand of $H^{0}(s_{\alpha}, H^{0}(s_{\alpha}w, V))=H^{0}(w, V)$. 

We now show that $\mathfrak{g}_{0, -\gamma}$ is a $B_{\gamma}$-summand of 
$H^{0}(s_{\alpha}, H^{0}(s_{\alpha}w, V))$. Let $S_{1}$ be the set of all simple roots $\beta$ such 
that $sl_{2, \beta}$ is a $B_{\beta}$-summand of $H^{0}(s_{\alpha}w, V)$. By the hypothesis, we have $\gamma \in S_{1}$. Since $S_{1}$ is a linearly independent subset of 
Hom$_{\mathbb{C}}(\oplus _{\beta \in S_{1}} \mathbb{C} h_{\beta}, \mathbb{C})$, there is a 
$h_{1}\in \oplus _{\beta \in S_{1}} \mathbb{C} h_{\beta}$ such that $\gamma(h_{1})=1$ and $\beta(h_{1})=0$ for every simple root $\beta$ in $S_{1}$ 
different from $\gamma$. If $\nu(h_{1})=0$ for every $\nu \in S\setminus S_{1}$, we are done. Otherwise, let $S_{2}$ be the set of all simple roots $\nu \in S\setminus S_{1} $ such 
that $\nu(h_{1})\neq 0$. Let $\nu\in S_{2}$. Then there is a 
$\beta\in S_{1}$ such that $\nu(h_{\beta})=-1$. Since $\beta \in S_{1}$, 
we have $h_{\beta}\in H^{0}(s_{\alpha}w, V)$. Hence, the Lie bracket 
$x_{-\nu}=[x_{-\nu}, h_{\beta}]$ is in $H^{0}(s_{\alpha}w, V)$. Hence, 
$\mathfrak{g}_{-\nu}$ must be a subspace of $H^{0}(s_{\alpha}w, V)$. 
Therefore, by induction applied to the simple root $\nu$, we see that 
either $\mathfrak{g}_{0, -\nu}$  or $sl_{2, \nu}$ is a 
 $B_{\nu}$-summand of  $H^{0}(s_{\alpha}w, V)$. Since 
$\nu\notin S_{1}$, we conclude that $\mathfrak{g}_{0, -\nu}$ is an indecomposable $B_{\nu}$-summand of $H^{0}(s_{\alpha}w, V)$.  Note that the fundamental dominant coweight $h(\nu)$ is in $\mathfrak{h}\bigcap \mathfrak{g}_{0, -\nu}$. 

Let $h_{2}=\sum_{\nu \in S_{2}}\nu(h_{1})h(\nu)$. Then we have $h(\gamma)=h_{1}-h_{2}$.
Therefore, $\mathbb{C} h(\gamma)$ is a 
$B_{\alpha}$-summand of $H^{0}(s_{\alpha}w, V)$. Hence, we see that 
$H^{0}(s_{\alpha}, \mathbb{C} h(\gamma))=\mathbb{C} h(\gamma)$ is a 
subspace of $H^{0}(s_{\alpha}, H^{0}(s_{\alpha}w, V))$.
Thus, we conclude that 
$\mathfrak{g}_{0, -\gamma}$ is a $B_{\gamma}$-summand 
of $H^{0}(s_{\alpha}, H^{0}(s_{\alpha}w, V))=H^{0}(w, V)$.  

{\it Case 3:}
We assume that $\langle \gamma , \alpha \rangle =0$. 
If $\mathfrak{g}_{0, -\gamma}$ is a $B_{\gamma}$-summand 
of $H^{0}(s_{\alpha}w, V)$, using $\langle \gamma , \alpha \rangle =0$
we see that it is also a $B_{\alpha}$-summand 
of $H^{0}(s_{\alpha}, H^{0}(s_{\alpha}w, V))$. For the same reason, the vector bundle on $L_{\alpha}/B_{\alpha}$ associated to the $B_{\alpha}$-module $\mathfrak{g}_{0, -\gamma}$ is trivial. Thus,  
$\mathfrak{g}_{0, -\gamma}$ is a $B_{\gamma}$-summand of 
$H^{0}(s_{\alpha}, H^{0}(s_{\alpha}w, V))=H^{0}(w, V)$. The proof  of the case when $sl_{2, \gamma}$
is a $B_{\gamma}$-summand of $H^{0}(s_{\alpha}w, V)$ is similar.

\end{proof}

We now deduce the following lemma as a consequence of Lemma 3.3.

\begin{lemma}
Let $w\in W$  and $V$ be a $B$-submodule of 
$\mathfrak{g}$ containing $\mathfrak{b}$. Then we have
$H^{i}(w, V)=0$ for every $i\geq 1$.
\end{lemma}
\begin{proof}
 The proof is by induction on $l(w)$.
If $l(w)=0$, we are done. Otherwise, we choose a simple root $\gamma\in S$ be 
such that $l(s_{\gamma}w)=l(w)-1$. By Lemma 3.3, there is a decomposition of 
the $B_{\gamma}$-module $H^{0}(s_{\gamma}w, V)$ such that every indecomposable 
$B_{\gamma}$-summand $V^{\prime}$ of $H^{0}(s_{\gamma}w , V)$ must be one of 
the 5 types given in Lemma 3.3.   In view of Lemma 2.4, any such $V^{\prime}$ is of the form 
$V^{\prime}=V^{\prime \prime }\otimes \mathbb{C}_{a \omega_{\gamma}}$ for some irreducible 
$\hat{L_{\gamma}}$-module $V^{\prime \prime }$ and an integer $a\in \{-1,  0 \}$. Hence, using Lemma 2.3, we conclude that $H^{i}(L_{\gamma}/B _{\gamma}, V^{\prime})$ is zero for every indecomposable 
$B_{\gamma}$-summand $V^{\prime}$ of  $H^{0}(s_{\gamma}w , V)$ 
and for every $i\geq 1$.
Thus, we see that 
$$ H^{i}(P_{\gamma}/B, H^{0}(s_{\gamma}w , V))=0  \hspace{1cm}  (1)   $$ 
for all $i\geq 1$.
By induction on $l(w)$, we have $H^{i}(s_{\gamma}w , V)$ is zero 
for all $i\geq 1$. Now, using  (1) and using the short exact 
sequence {\it SES} of $B$-modules, we conclude that 
$H^{i}(w, V)$ is zero for all $i\geq 1$. This completes 
the proof of lemma.
\end{proof}

We now prove the following.
\begin{lemma}
Let $w\in W$. Let $V_{1}$ be a $B$-submodule of $\mathfrak{g}$ containing $\mathfrak{b}$
and $V_{2}$ be a $B$-submodule of $V_{1}$ containing $\mathfrak{b}$. Then we have 
\begin{myenumerate}
\item
$H^{i}(w, V_{1}/V_{2})=0$  for every  $i\geq 1$. 
\item
The homomorphism $\Pi_{w}:H^{0}(w, V_{1}) \longrightarrow H^{0}(w, V_{1}/V_{2})$ of $B$-modules
induced by the natural homomorphism  $\Pi:V_{1} \longrightarrow 
V_{1}/V_{2}$ is surjective and  kernel of $\Pi_{w}$  is 
$H^{0}(w ,V_{2})$. 
\item The restriction map $r:H^{0}(w_{0}, \mathfrak{g}/\mathfrak{b})\longrightarrow  H^{0}(w, \mathfrak{g}/\mathfrak{b})$ is surjective.
\end{myenumerate}
\end{lemma}
\begin{proof}
{\it Proof of (1):}
We have the short exact sequence   
$0\longrightarrow V_{2} \longrightarrow V_{1} \longrightarrow 
V_{1}/V_{2} \longrightarrow 0$ of $B$-modules.
Applying $H^{0}(w, -)$ to this short exact sequence of $B$-modules, 
we obtain the following long exact sequence of $B$-modules: 
$$\cdots H^{i}(w, V_{2}) \longrightarrow H^{i}(w, V_{1}) 
\longrightarrow H^{i}(w, V_{1}/V_{2})\longrightarrow 
H^{i+1}(w, V_{2}) \cdots$$
By Lemma 3.4, $H^{i}(w, V_{1})$ and  
$H^{i+1}(w, V_{2})$ are zero for every $i\geq 1$.
Thus, we conclude that $H^{i}(w, V_{1}/V_{2})=0$
for every $i\geq 1$. This proves (1).

{\it Proof of (2):}
Taking $i=0$ in the above long exact sequence of $B$-modules and using 
$H^{1}(w, V_{2})=0$ ( see Lemma 3.4 ),  we obtain the following short exact sequence of 
$B$-modules: 
$$0\longrightarrow H^{0}(w, V_{2}) \longrightarrow H^{0}(w, V_{1}) 
\longrightarrow H^{0}(w, V_{1}/V_{2})\longrightarrow 0.$$
This proves (2).
  
{\it Proof of (3):}   We have the following commutative diagram of $B$-modules:
  
 \begin{center}  

$\xymatrix{ H^{0}(w_{0}, \mathfrak{g}) 
\ar[d]_{\Pi_{w_{0}}}\ar[r]^{res} & H^{0}(w, \mathfrak{g}) \ar[d]_{\Pi_{w}}\\ 
H^{0}(w_{0}, \mathfrak{g}/\mathfrak{b}) \ar[r]^{r} & H^{0}(w, \mathfrak{g}/\mathfrak{b}).}$

\end{center}

By (2), $\Pi_{w}:H^{0}(w, \mathfrak{g})\longrightarrow H^{0}(w, \mathfrak{g}/\mathfrak{b})$
is surjective. By Lemma 2.5(1), the restriction map $res:H^{0}(w_{0}, \mathfrak{g})\longrightarrow H^{0}(w, \mathfrak{g})$ is an isomorphism.  Thus, $r:H^{0}(w_{0}, \mathfrak{g}/\mathfrak{b})\longrightarrow  H^{0}(w, \mathfrak{g}/\mathfrak{b})$ is surjective. 

This completes the proof of the lemma.

\end{proof}

The following is a useful consequence of Lemma 3.5. 

\begin{corollary}
Let $w\in W$ and $\alpha\in R^{+}$.  Then we have
$H^{i}(w, \alpha)=0$ for every $i\geq 1$.
\end{corollary}
\begin{proof}
Let $V_{1}:=\oplus_{\mu\leq \alpha}\mathfrak{g}_{\mu}$ denote the 
direct sum of the weight spaces of $\mathfrak{g}$ of weights $\mu$ satisfying
$\mu\leq \alpha$ and let $V_{2}:=\oplus_{\mu < \alpha}\mathfrak{g}_{\mu}$ 
denote the direct sum of the weight spaces of $\mathfrak{g}$ of weights 
$\mu$ satisfying $\mu < \alpha$. 
It is clear that $V_{2}$ is a $B$-submodule of $\mathfrak{g}$ containing $\mathfrak{b}$ and $V_{1}$ is a $B$-submodule of $\mathfrak{g}$ containing $V_{2}$.
Since the $B$-module $V_{1}/V_{2}$ is isomorphic to $\mathbb{C}_{\alpha}$, we have 
$H^{i}(w, \alpha)= H^{i}(w, V_{1}/V_{2})$ for every $i\geq 1$.
Hence, by Lemma 3.5(1), $H^{i}(w, \alpha)=0$ for every 
$i\geq 1$.
This completes the proof of the corollary.
\end{proof}

\section {Main Results in the simply laced case} 
In this section, we prove the main results in the simply laced case.
Through out this section, we assume that $G$ is simply laced.

We first prove the following:

\begin{theorem}
Let $w\in W$ and let $\alpha_{0}$ denote the highest root of $G$ with 
respect to $T$ and $B^{+}$.  Then we have 
\begin{myenumerate}
\item
$H^{i}(X(w), T_{G/B})=0$ for every 
$i\geq 1$. 
\item
$H^{0}(X(w) , T_{G/B})$ is the adjoint representation 
$\mathfrak{g}$ of  $G$ if and only if  
$w^{-1}(\alpha_{0})$ is a negative root.
\end{myenumerate}
\end{theorem}
\begin{proof}
Since the tangent space of $G/B$ at the point $idB$ is 
$\mathfrak{g}/\mathfrak{b}$, the tangent bundle $T_{G/B}$ 
is the homogeneous vector bundle $\mathcal{L}(\mathfrak{g}/\mathfrak{b})$ 
on $G/B$ associated to the $B$-module $\mathfrak{g}/\mathfrak{b}$.

Hence, it is sufficient to prove the following:
\begin{myenumerate}
\item
$H^{i}(w, \mathfrak{g}/\mathfrak{b})=0$ for every $i\geq 1$. 
\item
$H^{0}(w , \mathfrak{g}/\mathfrak{b})$ is the adjoint representation 
$\mathfrak{g}$ of $G$ if and only if 
$w^{-1}(\alpha_{0}) < 0$. 
\end{myenumerate}

The proof  of (1) follows from Lemma 3.5(1).

To prove (2), we first note that 
the natural projection $\Pi:\mathfrak{g} \longrightarrow 
\mathfrak{g}/\mathfrak{b}$ of $B$-modules induces a 
homomorphism $\Pi_{w}:H^{0}(w, \mathfrak{g})
\longrightarrow H^{0}(w, \mathfrak{g}/\mathfrak{b})$ of $B$-modules.
Since the evaluation map $ev:H^{0}(w, \mathfrak{g})\longrightarrow 
\mathfrak{g}$ is an isomorphism ( see Lemma 2.5(1) ), using Lemma 3.5(2), we have the following 
short exact sequence of $B$-modules:
$$0\longrightarrow H^{0}(w, \mathfrak{b}) \longrightarrow \mathfrak{g} 
\longrightarrow H^{0}(w, \mathfrak{g}/\mathfrak{b})\longrightarrow 0.$$
Taking $-\alpha_{0}$-weight spaces, we obtain the following short eact 
sequence of vector spaces:
$$0\longrightarrow H^{0}(w, \mathfrak{b})_{-\alpha_{0}} \longrightarrow \mathfrak{g}_{-\alpha_{0}} \longrightarrow H^{0}(w, \mathfrak{g}/\mathfrak{b})_{-\alpha_{0}}\longrightarrow 0.$$
Since $\mathfrak{g}_{-\alpha_{0}}$ is one dimensional,  $H^{0}(w, \mathfrak{b})_{-\alpha_{0}}$ is zero if and only if $H^{0}(w, \mathfrak{g}/\mathfrak{b})_{-\alpha_{0}}$ is non-zero. 
Since $\mathfrak{g}$  is an irreducible $G$-module of highest weight $\alpha_{0},$ 
the $B$-stable line in $\mathfrak{g}$  is unique and it is the one dimensional subspace 
$\mathfrak{g}_{-\alpha_{0}}$. Therefore, it follows that  
$H^{0}(w, \mathfrak{b})$ is zero if and only if 
$H^{0}(w, \mathfrak{b})_{-\alpha_{0}}$ is zero.

We now show that $H^{0}(w, \mathfrak{g}/\mathfrak{b})_{-\alpha_{0}}$ is non-zero if and only if 
$w^{-1}(-\alpha_{0})\in R^{+}$.  For this, we first note that  the $B$-module 
$\mathfrak{g}/\mathfrak{b}$ has a composition series of $B$-modules with each successive simple quotient is isomorphic to $\mathbb{C}_{\alpha},$ where $\alpha$ is running over positive roots.  
By taking global sections $H^{0}(X(w),\mathcal{L}(w, \mathbb{C}_{\alpha}))$ 
and applying Corollary 3.6,  we see that the $B$-module $H^{0}(w, \mathfrak{g}/\mathfrak{b})$ has a filtration of $B$-modules with each successive quotient is isomorphic to $H^{0}(w, \alpha)$
for some positive root $\alpha$.  Therefore, it follows that $H^{0}(w, \mathfrak{g}/\mathfrak{b})_{-\alpha_{0}}$ is non-zero if and only if  $H^{0}(w, \alpha)_{-\alpha_{0}}$ is non-zero for some positive root $\alpha.$  On the other hand, for any positive root $\alpha$, there is a $v\in W$ such that $v(\alpha_{0})=\alpha$. Without loss of generality,  we may assume that $v$ is of minimum length among such elements.  It follows from the Demazure character formula that if  
$H^{0}(w , \alpha_{0} )_{\mu}\neq 0$, then $\mu$ is in the convex hull of the set $\{x(\alpha_{0}): x\leq wv\}$ ( see [4, Theorem 3.3.8, p.97, equation (3)] and [11, Proposition 14.18(b), p.379] ).  Note that  the convention for the signature of the weights in the Demazure character formula in this paper is  the same as the one in [11].  Further,  using  the above arguments, we see that $\mathbb{C}_{\alpha}$ is a $B$-submodule of $H^{0}(v, \alpha_{0})$. Therefore, by using {\it SES }, we see that $H^{0}(w , \alpha)$ is a $B$-submodule of $H^{0}(w v , \alpha_{0} )$.  Hence, every weight $\mu$ 
of  $H^{0}(w, \alpha)$ satisfies $\mu\geq w(\alpha)$. Clearly, $w(\alpha)\geq -\alpha_{0}$.  Hence, $H^{0}(w, \alpha)_{-\alpha_{0}}$ is non zero if and only if $w(\alpha)=-\alpha_{0}$. Thus, we conclude that  $H^{0}(w, \mathfrak{g}/\mathfrak{b})_{-\alpha_{0}}$ is non-zero if and only if $w^{-1}(-\alpha_{0})\in R^{+}$. 

Summarising the above arguments, we conclude that 
$H^{0}(w, \mathfrak{b})$ is zero if and only if $w^{-1}(\alpha_{0})$ is a 
negative root. Since $Ker(\Pi_{w})=H^{0}(w, \mathfrak{b})$ ( see Lemma 3.5(2) ),  we see that 
$Ker(\Pi_{w})$ is zero if and only if $w^{-1}(\alpha_{0})$ is a negative 
root. This completes the proof of (2).
\end{proof}

The following theorem describes the automorphism group of a smooth Schubert variety in 
the simply laced case.

\begin{theorem}
Let $w\in W$ be such that $X(w)$ is smooth. Let $A_{w}$ denote 
the connected component of the group of all automorphisms of $X(w)$ containing the  identity automorphism. For the left action of $G$ on $G/B$,  let  $P_{w}$ denote the stabiliser of 
$X(w)$ in $G.$ Then we have 
\begin{enumerate}
\item 
The homomorphism $\phi_{w}:P_{w} \longrightarrow A_{w}$ induced by the action of $P_{w}$ on $X(w)$ is surjective. 
\item
$\phi_{w}:P _{w}\longrightarrow A_{w}$ is an isomorphism if and only if  
$w^{-1}(\alpha_{0})$ is a negative root.
\end{enumerate}
\end{theorem}
\begin{proof}
Let $T_{w}$ denote the tangent bundle of $X(w)$. 
By [15, Theorem 3.7, p.17], we see that  $A_{w}$ is an algebraic group. 
Further, by [15, Lemma 3.4, p.13], it follows that the Lie algebra of $A_{w}$ is isomorphic to the 
space of all global sections $H^{0}(X(w), T_{w})$. Since $T_{w}$ is a vector subbundle of $T_{G/B}$, 
we have an injective homomorphism $i:H^{0}(X(w), T_{w})\hookrightarrow H^{0}(X(w), T_{G/B})=H^{0}(w, \mathfrak{g}/\mathfrak{b})$ of $B$-modules.  
We first note that $\phi_{w}$ induces a
homomorphism $\psi_{w}:\mathfrak{p}_{w}\longrightarrow H^{0}(X(w), T_{w})$
of Lie algebras, where $\mathfrak{p}_{w}$ is the Lie algebra of $P_{w}$. 

We now prove (1). By Lemma  3.5(3), the  restriction map
$r:\mathfrak{g}=H^{0}(w_{0}, \mathfrak{g}/\mathfrak{b})\longrightarrow
H^{0}(w, \mathfrak{g}/\mathfrak{b})$ is surjective. 
Since $P_{w}$ is a parabolic subgroup of $G$ containing $B$, the Lie algebra $\mathfrak{p}_{w}$ is 
a Lie subalgebra of $\mathfrak{g}=H^{0}(w_{0}, \mathfrak{g}/\mathfrak{b})$  containing $\mathfrak{b}$. 
Further, since $X(w)$ is a $P_{w}$-stable subvariety for the left action of $P_{w}$ on $G/B$,
we have the following commutative diagram of $B$-modules:

\begin{center}

$\xymatrix{ 
\mathfrak{p}_{w}\ar@{^{(}->}[d]\ar[r]^ {\psi_{w}} &
H^0(X(w), T_{w})\ar@{^{(}->}[d]^{i}
\\ H^0(w_{0}, \mathfrak g/\mathfrak b) \ar[r]^{r}  & H^0(w, \mathfrak g/ \mathfrak b )}
$

\end{center}

Now, let $\mathfrak{q}=r^{-1}(H^{0}(X(w), T_{w}))$. Note that since $\mathfrak{q}$ is a $B$-submodule  of $\mathfrak{g}$ containing $\mathfrak{p}_{w}$, $\mathfrak{q}$ is a parabolic subalgebra of $\mathfrak{g}$ containing $\mathfrak{p}_{w}$. We denote the restriction of $r$ to $\mathfrak{q}$ also by $r$.

We now show that $\mathfrak{p}_{w}=\mathfrak{q}$. Since $\mathfrak{g}=H^{0}(w_{0}, \mathfrak{g}/\mathfrak{b}),$ every element $x\in \mathfrak{q}\subset \mathfrak{g}$ is a tangent vector field on $G/B$. Further, by the definition of $\mathfrak{q},$ the ideal sheaf of $X(w)$  is $x$-stable for every $x\in \mathfrak{q}.$ Therefore, $\mathfrak{q}$ is contained in the Lie algebra $\mathfrak{p}_{w}$ of the stabiliser $P_{w}$ of $X(w)$ in $G$. Thus, we have $\mathfrak{p}_{w}=\mathfrak{q}.$  

Clearly,  $r:\mathfrak{q}\longrightarrow H^{0}(X(w), T_{w})$  is a homomorphism of Lie algebras. Therefore $\psi_{w}:\mathfrak{p}_{w}\longrightarrow H^{0}(X(w), T_{w})$ is surjective. By counting the dimensions of the images both at the level of algebraic groups and at the level of their Lie algebras, 
we conclude that $\phi_{w}:P_{w}\longrightarrow A_{w}$ is surjective. 

We now prove (2).

Assume that $w^{-1}(\alpha_{0})$ is a negative root.  Then by  the proof of Theorem 4.1(2),  the homomorphism $\Pi_{w}:H^{0}(w, \mathfrak{g})=\mathfrak{g}\longrightarrow H^{0}(w,\mathfrak{g}/\mathfrak{b})=\mathfrak{g}$ of $B$-modules induced by the natural homomorphism $\mathfrak{g}\longrightarrow \mathfrak{g}/\mathfrak{b}$  is an isomorphism.  Therefore,  $H^{0}(w, \mathfrak{g}/\mathfrak{b})$  has a unique $B$- stable line, namely $\mathfrak{g}_{-\alpha_{0}}$.  Note that since 
$w^{-1}(\alpha_{0}) < 0$, we have $w\neq 1$. Therefore, the action of $P_{w}$ on $X(w)$ is non trivial. Hence,  the homomorphism 
$\psi_{w}:\mathfrak{p}_{w}\longrightarrow H^{0}(X(w), T_{w})$ of $B$-modules is non-zero.  Therefore, the $B$-stable line $H^{0}(w, \mathfrak{g}/\mathfrak{b})_{-\alpha_{0}}$ is in the image 
$$\psi_{w}(\mathfrak{p}_{w})\subset H^{0}(X(w), T_{w})\subset H^{0}(w, \mathfrak{g}/\mathfrak{b}).$$
Hence, we have $\mathfrak{g}_{-\alpha_{0}}\bigcap ker(\psi_{w})=0$. Thus, 
$\psi_{w}:\mathfrak{p}_{w}\longrightarrow H^{0}(X(w), T_{w})$ is injective.  Since the base field is $\mathbb{C},$ it follows that the kernel of $\psi_{w}$ is the Lie algebra of the kernel of  $\phi_{w}.$   Therefore, $\phi_{w}:P_{w}\longrightarrow A_{w}$ is injective. Now, that
$\phi_{w}$ is an isomorphism follows from (1). This is again because the base field is $\mathbb{C}.$

Conversely, if $\phi_{w}:P_{w}\longrightarrow A_{w}$ is an isomorphism, then the induced homomorphism $\psi_{w}:\mathfrak{p}_{w}\longrightarrow H^{0}(X(w), T_{w})\subset H^{0}(w, \mathfrak{g}/\mathfrak{b})$ is injective. In particular, the $-\alpha_{0}$-weight space $H^{0}(w,\mathfrak{g}/\mathfrak{b})_{-\alpha_{0}}$
is non-zero. By using arguments similar to  the proof of Theorem 4.1,   we conclude that 
$w^{-1}(\alpha_{0})$ is a negative root.
\end{proof}

In the following Corollary, we describe the kernel of $\phi_{w}:P_{w} \longrightarrow A_{w}$. 
Let $J_{w}:=\{\alpha\in S: s_{\alpha}\leq w\}$ and 
let $T(w):=\bigcap_{\alpha\in J_{w}}Ker(\alpha)=\{t\in T: \alpha(t)=1 ~ for  ~ \alpha\in J_{w}\}$.
Let $U$  ( respectively, $U^{+}$ ) denote the unipotent radical of $B$ ( respectively, $B^{+}$ ) and  let $U_{-\alpha}$  ( respectively,  $U^{+}_{\alpha}$ ) be the the root subgroup of $U$ ( respectively, $U^{+}$ ) normalised by $T$ corresponding to $\alpha\in R^{+}$. Further, 
let $U_{\leq w}$ be the subgroup of $U$ generated by $\{U_{-\alpha}: \alpha\in  R^{+}\setminus (\bigcup_{v \leq w}R^{+}(v^{-1}))\}$. Then we have 

\begin{corollary}
The kernel $K_{w}$ of $\phi_{w}:P_{w} \longrightarrow A_{w}$ is generated by $T(w)$ and $U_{\leq w}$. 
 \end{corollary}
\begin{proof}
Let $R^{+}(w^{-1}):=\{\beta_{1}, \beta_{2}, \cdots, \beta_{l(w)}\}$.
Note that for every $1\leq j\leq l(w)$, we have $\beta_{j}=\sum_{\alpha \in S}m_{\alpha}\alpha$ where each $m_{\alpha}\in \mathbb{Z}_{\geq 0}$ and $m_{\alpha}=0$ unless $s_{\alpha}\leq w$. Therefore, it follows that $T(w) \subset Ker(\beta_{j})$ for every $1\leq j \leq l(w)$. Hence the action of $T(w)$ on $BwB/B=\Pi_{j=1}^{l(w)}U_{-\beta_{j}}wB/B$ is trivial. Here, we note that product is independent of the ordering  of $R^{+}(w^{-1})$ ( see [11, II, p.354] ). Since $BwB/B$ is an open dense subset of $X(w),$ we conclude that the action of $T(w)$ on $X(w)$ is trivial.

If the action of $U_{-\beta}$ on $X(w)$ is trivial for some positive root $\beta$, then it fixes $vB/B$ 
for every $v\leq w$. Hence, $v^{-1}(\beta)\in R^{+}$ for every $v\leq w.$ 

Conversely, if  $\beta\in R^{+}$ is such that $v^{-1}(\beta) > 0$ for every $v\leq w$, then we show that the action of $U_{-\beta}$ on $X(w)$ is trivial. We show this by induction on $l(w)$. So, fix such a $\beta$.

If $l(w)=1$, then we have $w=s_{\alpha}$ for some simple root $\alpha$. For any $u_{1}\in U_{-\alpha}$ and $u_{2}\in U_{-\beta}$, the commutator $u_{2}^{-1}u_{1}^{-1}u_{2}u_{1}\in \Pi_{\gamma}U_{-\gamma}$, where the product is taken over all positive roots $\gamma$ of the form $\gamma=i\alpha+j\beta$ with $i, ~  j\in \mathbb{N}$ (see, [10, p.203] and [10, p.209 -  p.215] ). First note that $\alpha\neq \beta,$ since $s_{\alpha}(\beta)\in R^{+}.$ Further, if $u_{1}$ and $u_{2}$ commute with each other, then $u_{1}wB/B$ is fixed by $u_{2}$. Therefore, we may assume that they do not commute with each other. 

Since $G$ is simply laced, we have  $\langle \alpha , \beta \rangle=\langle \beta , \alpha \rangle =-1$. Therefore, it follows that 
$\alpha+\beta=s_{\alpha}(\beta)\in R^{+}$  and $u_{2}^{-1}u_{1}^{-1}u_{2}u_{1}\in U_{-(\alpha+\beta)}$. It is easy to see that the subgroup $U_{-(\alpha+\beta)}$ fixes the point $s_{\alpha}B/B$. Hence, by the above arguments, it follows that $$u_{2}u_{1}wB/B=u_{1}u_{2}(u_{2}^{-1}u_{1}^{-1}u_{2}u_{1})wB/B=u_{1}u_{2}wB/B=u_{1}wB/B.$$

We may assume that $l(w)\geq 2$  and choose a simple root $\alpha$ such that $l(s_{\alpha}w)=l(w)-1$.  Let $v=s_{\alpha}w$ and let $n_{\alpha}$ be a representative of $s_{\alpha}$ in $N_{G}(T).$  We have $BwB/B=U_{-\alpha}n_{\alpha}U_{v}vB/B$, where $U_{v}=\Pi_{\gamma\in R^{+}(v^{-1})}U_{-\gamma}$ ( see [11, II, p.354] ).  Note that by induction, the action of $U_{-\beta}$ on 
$U_{v}vB/B$ is trivial. Therefore,  by the above arguments we may assume that the subgroups $U_{-\alpha}$ and $U_{-\beta}$ do not commute with each other. By the proof in the case of length one, for every $u_{2}\in U_{-\beta},$ $u_{1}\in U_{-\alpha}$ and $u\in U_{v}$, we see that $$u_{2}\cdot (u_{1}n_{\alpha}uvB/B)=u_{1}n_{\alpha}\cdot (u^{\prime}\cdot uvB/B),$$ for some 
$u^{\prime}\in U_{-s_{\alpha}(\beta)}U_{-\beta}$. Note that for every $v_{1}\leq v$, both $v_{1}^{-1}(\beta)$ and 
$v_{1}^{-1}(s_{\alpha}(\beta))$ are positive roots. Therefore, by the induction hypothesis, $u^{\prime}$ fixes the point $uvB/B$.
Thus, we conclude that the action of $U_{-\beta}$ on the Schubert cell $BwB/B$ is trivial. Since, $BwB/B$ is an open dense subset of $X(w)$, it follows that the action of $U_{-\beta}$ on $X(w)$ is trivial.    

Finally, we show that $U^{+}_{\beta}\bigcap K_{w}=\{1\}$ for every positive root $\beta$ such that $U^{+}_{\beta}$ is a subgroup of $P_{w}$. First note that if  $U^{+}_{\beta}\bigcap K_{w}$ is non trivial, then by conjugating by $T$, we could prove that $U^{+}_{\beta}\subset K_{w}.$   Further, since $G$ is simply laced, if $\beta$ is not a simple root, then there is a simple root $\alpha$ such that  $\langle \beta , \alpha \rangle =1.$  Therefore, by using similar arguments as above, we conclude that  $1\neq u_{2}^{-1}u_{1}^{-1}u_{2}u_{1}\in U^{+}_{\beta-\alpha},$ for some $u_{1}\in U_{-\alpha}$ and some $u_{2}\in U^{+}_{\beta}.$ Hence, $U^{+}_{\beta-\alpha}\bigcap K_{w}$ is non trivial. Therefore, $U^{+}_{\beta-\alpha}\subset K_{w}.$ Proceeding this way, we are able to find 
a simple root $\gamma$  such that $U^{+}_{\gamma}\subset K_{w}.$  Hence, by using the arguments 
similar to the proof of  Theorem 4.2, we conclude that $K_{w}$ contains a representative $n_{\gamma}$  of  $s_{\gamma}$ in $N_{G}(T).$  Thus, we have $l(s_{\gamma}w)=l(w)-1.$ This is a contradiction to the fact that if $\alpha$ is a simple root such that  $l(s_{\alpha}w)=l(w)-1,$ then the action of  
$U^{+}_{\alpha}$ on $X(w)$ is non trivial.   

This completes the proof .
\end{proof}

Again we assume that  $X(w)$ is smooth.
Let $\mathcal{L}_{\alpha_{0}}$ denote the line bundle on 
$X(w)$ associated to $\alpha_{0}$. 
Consider the left action of $T$ on $G/B$. Note that $X(w)$ is 
stable under $T$. Let $\lambda$ be a dominant character of $T$. In the 
following Corollary, we use the notion of semi-stable points introduced  
by Mumford ( see [17]  ). We denote by $X(w^{-1})_{T}^{ss}(\mathcal{L}_{\alpha_{0}})$ the set of all semi-stable points of $X(w^{-1})$ with respect to the $T$-linearised 
line bundle $\mathcal{L}_{\alpha_{0}}.$  

Then we have

\begin{corollary}
Let $A_{w}$,  $P_{w}$  and 
$\phi_{w}:P_{w} \longrightarrow A_{w}$  be as in the hypothesis of  Theorem 4.2. 
Then $\phi_{w}:P_{w} \longrightarrow A_{w}$ is an isomorphism  if and only if the set 
$X(w^{-1})_{T}^{ss}(\mathcal{L}_{\alpha_{0}})$ of semi-stable points is non-empty.
\end{corollary}
\begin{proof}
By Theorem 4.2, we see that $\phi_{w}:P_{w} \longrightarrow A_{w}$ is an isomorphism 
if and only if  $w^{-1}(\alpha_{0})$ is a negative root. By [13, Lemma 2.1], we note that
$w^{-1}(\alpha_{0}) < 0$ if and only if 
$X(w^{-1})_{T}^{ss}(\mathcal{L}_{\alpha_{0}})$ is non empty.
This completes the proof of the corollary.
\end{proof}

The following Corollary connects the set 
$X(w^{-1})_{T}^{ss}(\mathcal{L}_{\alpha_{0}})$ of semi-stable points with the 
vanishing of all cohomology modules of the vector bundle 
$\mathcal{L}(w, \mathfrak{b})$ on $X(w)$.  
\begin{corollary}
The set $X(w^{-1})_{T}^{ss}(\mathcal{L}_{\alpha_{0}})$ of semi-stable points is 
non-empty if and only if $H^{i}(w, \mathfrak{b})=0$ for 
every $i\in \mathbb{Z}_{\geq 0}$.
\end{corollary}
\begin{proof}
By Lemma 3.4, $H^{j}(w, \mathfrak{b})=0$ for every $j\geq 1$.
By using the proof of Theorem 4.1, we see that $H^{0}(w, \mathfrak{b})$ is 
zero if and only if $w^{-1}(\alpha_{0})$ is a negative root.  Therefore, by using Theorem 4.2, it follows that $\phi_{w}:P_{w}\longrightarrow A_{w}$ is an isomorphism if and only if $H^{0}(w, \mathfrak{b})$ is 
zero. Now, the proof of the corollary follows from Corollary 4.4.
\end{proof}

\begin{remark}
We recall the restriction map $r: H^0(w_{0}, \mathfrak g/\mathfrak b) \longrightarrow H^0(w, \mathfrak g/ \mathfrak b )$ which is used in the proof of Theorem 4.2.  Let $\mathfrak{q}=r^{-1}(H^{0}(X(w), T_{w})).$  The proof of the statement $\mathfrak{q}=\mathfrak{p}_{w}$ is due to the referee.
We are very  grateful to him/her for the proof.
\end{remark}

\begin{remark}
Corollary 4.3 does not hold in the non simply laced case. For instance, 
let $G$ be of type $B_{2},$ and $w=s_{2}s_{1}.$ Here, we follow the convention
with $\langle \alpha_{1} , \alpha_{2} \rangle =-2$ and $\langle \alpha_{2} , \alpha_{1} \rangle =-1.$ 
In this case, we have $\beta=\alpha_{1}+\alpha_{2}\in R^{+}\setminus (\bigcup_{v\leq w} R^{+}(v^{-1})),$ and the natural action of $U_{-\beta}$ on $X(w)$ is not trivial. 
\end{remark}

\section { Cohomology modules in the non simply laced  case}

Throughout this section, we assume that $G$ is a simple  
algebraic group of adjoint type over $\mathbb{C}$ which is not simply laced.
Since $G$ is not simply laced, there is a highest long root and there is a highest short root. We denote the highest long root  by  $\alpha_{0}$ and the highest short root  by $\beta_{0}$.    

We first prove the following.

\begin{lemma}
There is a positive root $\beta$ and a simple root $\alpha$ such 
that $s_{\alpha}\circ  \beta=\beta_{0}$. 
\end{lemma}
\begin{proof}

Since $\beta_{0} < \alpha_{0}$, there is a simple root 
$\alpha$ such that $\beta_{0}+\alpha$ is a positive root. 
Since $\beta_{0} +\alpha \neq  \alpha $,  it follows that  $\beta=s_{\alpha}\circ\beta_{0}=s_{\alpha}(\beta_{0}+\alpha)$ is a positive root.
\end{proof}

Let $\alpha$ and $\beta$ be as in Lemma 5.1. 
Let $w\in W$ be such that $s_{\alpha}\leq w$. 
Let $V_{1}:=\oplus_{\mu\leq \beta}\mathfrak{g}_{\mu}$ be the direct 
sum of all weight spaces of weights $\mu$ satisfying $\mu\leq\beta$. Also, 
let $V_{2}:=\oplus_{\mu < \beta}\mathfrak{g}_{\mu}$ be the direct sum of all
weight spaces of weights $\mu$ satisfying $\mu < \beta$.
We note that $V_{2}$ is a $B$- submodule of $\mathfrak{g}$ containing $\mathfrak{b}$ and $V_{1}$ is a $B$-submodule of $\mathfrak{g}$ containing $V_{2}$. 

Then we have
\begin{lemma}
$H^{1}(w , V_{1}/V_{2})$  is non-zero.  
\end{lemma}
\begin{proof}
Note that  $\beta_{0}=s_{\alpha}\circ \beta$  
is a dominant character of $T$.  
Hence, by the Borel-Weil-Bott's theorem [3], 
$H^{1}(w_{0}, \beta)$ is an irreducible representation of $G$
with highest weight $s_{\alpha}\circ \beta$. Also, by Lemma 2.2, $H^{1}(s_{\alpha}, \beta)$ is non-zero. Also, by [12, Corollary 4.3], the restriction map $H^{1}(w_{0}, \beta) \longrightarrow 
H^{1}(s_{\alpha}, \beta)$ is surjective. 

Thus, the restriction map $H^{1}(w,  \beta)
\longrightarrow H^{1}(s_{\alpha}, \beta)$ is also 
surjective. Hence, $H^{1}(w,  \beta)$ is non-zero.
Since the quotient $B$-module $V_{1}/V_{2}$ is isomorphic to $\mathbb{C}_{\beta}$, we have 
$H^{1}(w,  V_{1}/V_{2})=H^{1}(w,  \beta)$. Hence 
$H^{1}(w,  V_{1}/V_{2})$ is non-zero.
This completes the proof.

\end{proof}

Let $V(\beta_{0})$ be the irreducible $G$-module of highest 
weight $\beta_{0}$.

{\it Notation:}
\begin{enumerate}
\item
Let $\gamma$ be a  short simple root. Since $V(\beta_{0})_{\gamma}$ is non zero, there is a $v\in V(\beta_{0})_{0}$ such that $x_{-\gamma}\cdot v$ is non zero and $x_{-\nu}\cdot v$ is zero for every 
simple root $\nu$ different from $\gamma$. The cyclic $B_{\gamma}$-submodule of $V(\beta_{0})$ generated by $v$  is a two dimensional indecomposable $B_{\gamma}$-submodule 
of type similar to that of type (2) in the statement of Lemma 3.3.  We denote it by $V(\beta_{0})_{0, -\gamma}$. 
\item
Let $\gamma$ be a simple root not necessarily a short root. If $\beta$ is a short root such that $\langle \beta , \gamma \rangle =1$,
then the cyclic $B_{\gamma}$-submodule of $V(\beta_{0})$ generated by the one dimensional subspace $V(\beta_{0})_{\beta}$ is a two dimensional  indecomposable $B_{\gamma}$-submodule of type similar to that of type (4) in the statement of 
Lemma 3.3. We denote it by $V(\beta_{0})_{\beta , \beta - \gamma}$. 
\end{enumerate}

Let $V$ be a $B$-submodule of $V(\beta_{0})$ such 
that for every short simple root $\alpha$, there is a 
decomposition of the $B_{\alpha}$-module $V$ such that every indecomposable  
$B_{\alpha}$-summand $V^{\prime}$ of $V$ is one of the following:  

\begin{myenumerate}
\item
$V^{\prime}=\mathbb{C} v$ for some non zero vector $v\in V(\beta_{0})_{0}$ such that 
$x_{-\alpha}\cdot v=0.$ 
\item  
$V^{\prime}=V(\beta_{0})_{0, -\alpha}$.
\item
$V^{\prime}=V(\beta_{0})_{\beta}$ for some short root $\beta$ 
such that  $\langle \beta , \alpha \rangle \in \{-1, 0\}$. 
\item  $V^{\prime}= V(\beta_{0})_{\beta , \beta-\alpha}$ 
for some short root $\beta$ such $\langle \beta , \alpha \rangle =1.$ 
\item
$V^{\prime}=sl_{2,\alpha}$, the three dimensional irreducible 
$L_{\alpha}$-module with highest weight $\alpha$.
\end{myenumerate}

The following lemma is useful to prove the main results in section 5 and section 6.

 \begin{lemma}
Let $\phi\in W$ and $\gamma\in S$. Then we have
\begin{myenumerate}
\item
If  $\gamma$ is a short root, then there is a decomposition of the 
$B_{\gamma}$-module $H^{0}(\phi, V)$ such that every indecomposable 
$B_{\gamma}$-summand of $H^{0}(\phi, V)$ is one of the five types as in the 
hypothesis above.
\item If $\gamma$ is a long root, then there is a decomposition of the $B_{\gamma}$-module $H^{0}(\phi, V)$ such that 
every indecomposable $B_{\gamma}$-summand of $H^{0}(\phi, V)$ is one of the types (1), (3) or (4) as in the hypothesis above.
\end{myenumerate}
\end{lemma}

\begin{proof}
{ \it Proof of (2):}  If $\gamma$ is a long simple root, then $V(\beta_{0})_{-\gamma}=0$.  The proof  of (2) follows by  using similar arguments in the proof of  Lemma 3.3. 

The proof  of (1)  is essentially the same as that of Lemma 3.3. 
\end{proof}

\begin{corollary}
Let $\phi\in W$.  
\begin{enumerate}
\item
Let $V$ be a $B$-submodule of $V(\beta_{0})$ satisfying the 
hypothesis of Lemma 5.3. Then we have 
$H^{i}(\phi, V)=0$ for every $i\geq 1$.
\item
Let $V$ be a $B$-submodule of $V(\beta_{0})$ such that $V_{\mu}=0$
unless $\mu\in -(R^{+}\setminus S)$. Then we have 
$H^{i}(\phi, V)=0$ for every $i\geq 1$.
\end{enumerate}
\end{corollary}
\begin{proof}
The proof is by induction on $l(\phi)$.
If $l(\phi)=0$, then  $\phi=1$ and so we are done.
Otherwise, choose a simple root $\alpha$ such that 
$l(\phi)=1+l(s_{\alpha}\phi)$.
By induction, we have $H^{i}(s_{\alpha}\phi, V)=0$ for every $i\geq 1.$ 
Applying this to {\it SES}, we see that 
$H^{i}(\phi, V)=0$ for every $i\geq 2$ and  
$H^{1}(\phi, V)=H^{1}(s_{\alpha}, H^{0}(s_{\alpha}\phi, V)).$ 
Therefore, it is sufficient to show that $H^{1}(s_{\alpha}, H^{0}(s_{\alpha}\phi, V))=0.$ 

{\it Proof of 1:}  If $\alpha$ is a long root, then  by Lemma 5.3(2), every indecomposable 
$B_{\alpha}$-summand $V^{\prime}$ of $H^{0}(s_{\alpha}\phi, V)$ is one of 
the types (1), (3) or (4) as in the hypothesis of Lemma 5.3. In all these cases, 
in view of Lemma 2.4, we see that there is an irreducible $\hat{L_{\alpha}}$-module 
$V^{\prime \prime}$ and an integer $a\in \{-1,0\}$ such that   
$V^{\prime}=V^{\prime \prime}\otimes \mathbb{C}_{a\omega_{\alpha}}.$ Hence, by using Lemma 2.3,
we see that $H^{i}(s_{\alpha}, V^{\prime})=0$ for every $i\geq 1.$
Thus, we have $H^{1}(s_{\alpha}, H^{0}(s_{\alpha}\phi, V))=0.$

If $\alpha$ is a short root, using Lemma 5.3(1), we see that 
every indecomposable $B_{\alpha}$-summand $V^{\prime}$ of 
$H^{0}(s_{\alpha}\phi, V)$ is one of the five types as in the hypothesis 
of Lemma 5.3. Therefore, in view of Lemma 2.4, we see that every 
indecomposable $B_{\alpha}$-summand $V^{\prime}$ of 
$H^{0}(s_{\alpha}\phi, V)$ is of the form 
$V^{\prime}=V^{\prime \prime}\otimes \mathbb{C}_{a\omega_{\alpha}}$ for some irreducible 
$L_{\alpha}$-module $V^{\prime \prime}$ and an integer $a\in \{-1,0\}.$
We apply Lemma 2.3 to conclude that  $H^{1}(s_{\alpha}, H^{0}(s_{\alpha}\phi, V))=0$.

This completes the proof of (1).

{\it Proof of 2:}  By Lemma 2.6, $H^{0}(s_{\alpha}\phi , V)$ is a $B$-submodule of $V$ and therefore 
it satisfies the hypothesis of (2).  Let  $V^{\prime}$ be an indecomposable $B_{\alpha}$-summand  of  $H^{0}(s_{\alpha}\phi, V)$. In view of Lemma 2.4, we have  $V^{\prime}=V^{\prime \prime}\otimes \mathbb{C}_{a\omega_{\alpha}}$ for some irreducible $\hat{L_{\alpha}}$-module $V^{\prime \prime}$ and an integer $a$.  
In particular, the set of all weights of $V^{\prime}$ is equal to $\{\lambda+a\omega_{\alpha}, \lambda-\alpha+a\omega_{\alpha}, \cdots , s_{\alpha}(\lambda)+a\omega_{\alpha} \}$,  where $\lambda$ is the 
highest weight of the irreducible $L_{\alpha}$-module $V^{\prime \prime}$.  
Note that by the hypothesis, we have $$\lambda -i\alpha+a\omega_{\alpha}\in -(R^{+}\setminus S),$$  
for every $0 \leq i \leq \langle \lambda , \alpha \rangle $. 
Further,  $\lambda -i\alpha+a\omega_{\alpha}$ is a short root for every $0 \leq i \leq \langle \lambda , \alpha \rangle $. Therefore, we have  $$-1\leq \langle s_{\alpha}(\lambda) + a\omega_{\alpha}, \alpha \rangle \leq  a \leq \langle \lambda + a\omega_{\alpha}, \alpha \rangle \leq  1. $$ 
Thus, we have $a\in \{-1,0, 1 \}$. Hence, by using Lemma 2.3,
we see that $H^{i}(s_{\alpha}, V^{\prime})=0$ for every $i\geq 1.$
Thus, we have $H^{1}(s_{\alpha}, H^{0}(s_{\alpha}\phi, V))=0.$

This completes the proof of (2).

\end{proof}

\begin{corollary}
Let $\phi\in W$.  Let $V_{1}$ be a $B$-submodule of $V(\beta_{0})$ and $V_{2}$ be a 
$B$-submodule of  $V_{1}$ .  Assume that both $V_{1}$ and $V_{2}$ satisfy either the 
hypothesis of Corollary 5.4(1)  or the hypothesis of Corollary 5.4(2).
Then we have 
$H^{i}(\phi, V_{1}/V_{2})=0$ for every $i\geq 1$.
\end{corollary}
\begin{proof}

By Corollary 5.4, we see that  
$H^{i}(\phi, V_{j})=0$ for every $i\geq 1$ and $j=1,2.$ 
Now, the proof follows by using the arguments similar to the proof of Lemma 3.5(1).
\end{proof}

\begin{corollary}
Let $\phi\in W$ and $\alpha$ be a short root such that $-\alpha\notin S$. \\
Then we have 
$H^{i}(\phi, \alpha)=0$ for every $i\geq 1$.
\end{corollary}
\begin{proof}
Take $V_{1}:=\oplus_{\mu\leq \alpha}V(\beta_{0})_{\mu}$ and 
$V_{2}:=\oplus_{\mu < \alpha}V(\beta_{0})_{\mu}$.  If $\alpha\in R^{+}$, 
then $\oplus_{\mu\leq 0}V(\beta_{0})_{\mu} $ is a  subspace of $V_{2}$. 
Therefore, imitating the proof of Lemma 3.3, we see that both the $B$-modules $V_{1}$ and $V_{2}$ satisfy the hypothesis of Lemma 5.3 ( the only difference here is that we have to work with 
the zero wight space of $V(\beta_{0})$ instead of $\mathfrak{h}$ and the set of all short simple roots instead of the set of all simple roots ). Therefore, both $V_{1}$ and $V_{2}$ satisfy the hypothesis of Corollary 5.4(1).

Otherwise, both the $B$-modules $V_{1}$ and $V_{2}$ satisfy the hypothesis of 
Corollary 5.4(2). In both cases, the proof of the corollary follows by using Corollary 5.5 and 
imitating the proof of Corollary 3.6.
\end{proof}

\section{Main results in the non simply laced case}
In this section, we  prove the main results in the non simply laced case.
Through out this section, we assume that $G$ is not simply laced.

The following lemma is useful to prove the main results in this section. 

\begin{lemma} 
 Let $V$ be a $B$-module. Let $w\in W$ and let $\gamma$ be a simple root
such that $l(ws_{\gamma})=l(w)-1$. Let $\phi=ws_{\gamma}$.  Then
we have the following long exact sequence of $B$-modules:
\begin{align*}
&H^{1}(\phi , H^{0}(s_{\gamma}, V))\longrightarrow 
 H^{1}(w , V)\longrightarrow H^{0}(\phi, H^{1}(s_{\gamma}, V)) \longrightarrow  \cdots
 \\ 
& \cdots 
\longrightarrow H^{i}(\phi, H^{0}(s_{\gamma}, V)) \longrightarrow 
H^{i}(w, V) \longrightarrow 
H^{i-1}(\phi, H^{1}(s_{\gamma} , V)) \longrightarrow  \cdots
\end{align*}
\end{lemma}
\begin{proof}

Fix a decomposition of $V=\oplus_{j=1}^{r}V_{j}$ into a direct sum of indecomposable 
$B_{\gamma}$-modules. By using Lemma 2.4, for each $1\leq j \leq r$, there exists an irreducible 
$\hat{L_{\gamma}}$-module $V_{j}^{\prime}$ and an integer $m_{j}$  such that  $V_{j}=V_{j}^{\prime}\otimes \mathbb{C}_{m_{j}\omega_{\gamma}}$. Now, let $J_{1}:=\{j\in \{1, 2, \cdots, r \}: m_{j}\geq 0\}$
 and  $J_{2}:=\{1, 2, \cdots, r\} \setminus J_{1}$. By using Lemma 2.3, we see that the 
image $V_{+}$ of the evaluation map $H^{0}(s_{\gamma} , V)\longrightarrow V$ 
is equal to $\oplus_{j\in J_{1}}V_{j}$.

 Let $V_{-}:=\oplus_{j\in J_{2}}V_{j}$. Clearly, $V_{+}$ is a $B$-submodule of $V$ and $V_{-}$ is a $B_{\gamma}$-submodule of $V$. Note that  the restriction to $V_{-}$ of the natural map 
$V\longrightarrow V/V_{+}$ is an isomorphism of $B_{\gamma}$-modules.

Further, we have $H^{i}(s_{\gamma} , V_{+})=0$ for every 
$i\geq 1$ and $H^{i}(s_{\gamma} , V/V_{+})=0$ for every $i\neq 1$ ( see Lemma 2.3 ).
Also, the inclusion $V_{+}\subset V$ of $B$-modules induces an isomorphism 
$H^{0}(s_{\gamma} , V_{+})\longrightarrow H^{0}(s_{\gamma}, V)$ of $B$-modules.
Similarly, the natural map $V\longrightarrow V/V_{+}$ induces an isomorphism 
$H^{1}(s_{\gamma} , V)\longrightarrow H^{1}(s_{\gamma}, V/V_{+})$
of $B$-modules.

Let $w=s_{i_{1}}s_{i_{2}}\cdots s_{i_{r}}$ be a reduced  expression for $w$ with $s_{i_{r}}=s_{\gamma}$. Let $\underline i=(i_{1}, i_{2}, \cdots, i_{r})$ and $\underline i'=(i_{1}, i_{2}, \cdots, i_{r-1})$

Recall the morphism $f_r : Z(w, \underline i) \longrightarrow Z(\phi,
\underline i')$ from section 2. Using (Iso), we see that 
\[ R^{j}{f_{r}}_{*}{\mathcal L}(w, \underline i, V) = {\mathcal L}({\phi},  \underline i' , 
H^{j}(P_{\alpha_{i_r}}/B, {\mathcal L}(s_{i_r}, V))) \hspace{1.5cm}( j\geq 0). \] 

Using this isomorphism and the fact that  $H^{j}(s_{\gamma}, V_{+})=0$ for $j\geq 1$, it follows that  
the $B$-modules $H^{i}(\phi, H^{0}(s_{\gamma}, V_{+}))$ and $H^{i}(w, V_{+})$ are isomorphic 
for every $i\geq 0$. Similarly, the $B$-modules 
$H^{i-1}(\phi, H^{1}(s_{\gamma}, V/V_{+}))$ and 
$H^{i}(w, V/V_{+})$ are isomorphic for every $i\geq 1$. 
Summarising these observations, we see that the $B$-modules 
$H^{i}(\phi, H^{0}(s_{\gamma}, V))$ and $H^{i}(w, V_{+})$ are isomorphic for 
every $i\geq 0$. Similarly, we also see that 
$H^{i-1}(\phi, H^{1}(s_{\gamma}, V))$ and $H^{i}(w, V/V_{+})$ are 
isomorphic for every $i\geq 1$. 

Applying $H^{0}(w, -)$ to the short exact sequence 
$0 \longrightarrow V_{+} \longrightarrow V \longrightarrow V/V_{+} \longrightarrow 0$ of $B$-modules, we obtain the following long exact sequence of 
$B$-modules:
$$\cdots H^{i-1}(w, V/V_{+}) \longrightarrow H^{i}(w, V_{+}) \longrightarrow 
H^{i}(w, V) \longrightarrow H^{i}(w, V/V_{+})\longrightarrow 
H^{i+1}(w, V_{+}) \cdots$$

By using the above isomorphisms in this long exact sequence,  we obtain the following long
exact sequence of $B$-modules: 
\begin{align*}
&H^{1}(\phi , H^{0}(s_{\gamma}, V))\longrightarrow 
 H^{1}(w , V)\longrightarrow H^{0}(\phi, H^{1}(s_{\gamma}, V)) \longrightarrow  \cdots
 \\ 
& \cdots 
\longrightarrow H^{i}(\phi, H^{0}(s_{\gamma}, V)) \longrightarrow 
H^{i}(w, V) \longrightarrow 
H^{i-1}(\phi, H^{1}(s_{\gamma} , V)) \longrightarrow  \cdots
\end{align*}
This completes the proof of the lemma.

\end{proof}

We now deduce the following.
\begin{lemma}
Let $w\in W$  and  $V$ be a $B$-submodule of $\mathfrak{g}$.
Then we have $H^{i}(w, V)=0$ for every $i\geq 2$.
\end{lemma}
\begin{proof}
The proof is by induction on $l(w)$.
If $l(w)=0$, we are done. Otherwise, choose a simple root $\gamma\in S$ 
such that $l(ws_{\gamma})=l(w)-1$.

{\it Step1.} We show that  $H^{1}(s_{\gamma}, V)$ is either zero  or is a 
direct sum of the trivial $B$-module $\mathbb{C}_{0}$  with multiplicity at most 
one and a $B$-module $V^{\prime}$  having a composition series $V^{\prime}=V_{m} \supset  V_{m-1} \supset \cdots V_{1} \supset V_{0}$ such that  each subquotient  $V_{i+1}/ V_{i}$ ( $0\leq i \leq m-1$ )
is isomorphic to $\mathbb{C}_{\beta_{i}}$ for some short root $\beta_{i}$ which is not in $-S$. 

Assume that $H^1(s_{\gamma},  V)_{\mu}\neq 0$. Then there exists an
   indecomposable $B_{\gamma}$-direct summand $V_1$ of $V$ such that 
   $H^1(s_{\gamma}, V_1)_{\mu}\neq 0$. By Lemma 2.4, $V_1=V'\otimes \mathbb C_{a\omega_{\gamma}}$ for some irreducible $\hat{L_{\gamma}}$-module $V'$ and an integer $a$.
  Since $H^1(s_{\gamma}, V_1)\neq 0$, by Lemma 2.3, we have $a\leq -2$ and
  $H^1(s_{\gamma}, V_1)=V'\otimes H^0(s_{\gamma}, a\omega_{\gamma}-(a+1)\gamma)$. Therefore, 
  any weight $\mu''$ of $H^1(s_{\gamma}, V_1)$ is in the $\gamma$-string from $\mu_1+\gamma=
  \mu_1+\rho-s_{\gamma}(\rho)=s_{\gamma}(s_{\gamma}\circ \mu_1)$ to 
  $s_{\gamma}\circ \mu_1$, where $\mu_1$ is the lowest weight of $V_1$. Thus,  there is an integer  $1\leq t \leq  -(\langle \mu_{1} , \gamma \rangle +1)$ such that  $\mu=\mu_{1}+t\gamma$.
  
We now analyse the cases. 

If $\mu=0$,  then  $\mu_{1}=-\gamma$,  and the  indecomposable $B_{\gamma}$-summand  $V_{1}$ is  equal to $\mathfrak{g}_{-\gamma}$.  Further, by Lemma 2.2(3),  it follows that the trivial $B$-module $\mathbb{C}_{0}$ is a $B$-summand of  $H^{1}(s_{\gamma} , V)$.  Also,  the zero weight space $H^{1}(s_{\gamma}, V)_{0}$ of $H^{1}(s_{\gamma}, V)$ has multiplicity one in this case.

Otherwise,  by using the above discussion, we  see that $\mu_{1}$ is a root different from $-\gamma$
such that $\langle \mu_{1} , \gamma \rangle \in \{-2, -3\}$, and $V_{1}=\mathfrak{g}_{\mu_{1}}$.  In particular, $\mu$ is a short root.  Further, if  $\mu\in -S$, there is an integer $1\leq t \leq  -(\langle \mu_{1} , \gamma \rangle +1)$ such that  $\mu=\mu_{1}+t\gamma\in -S$. If $\langle \mu_{1} , \gamma \rangle =-2$, then  $t=1$ and so the simple roots $-\mu$ and $\gamma$ are orthogonal with 
$\mu_{1}=\mu -\gamma$ is a root, contradiction to the fact that sum of two orthogonal simple 
roots is not a root.  If  $\langle \mu_{1} , \gamma \rangle =-3$, then $t=1, 2$. The proof  of the case 
when $t=1$ follows by similar arguments as above. If  $t=2$,  then  $-\mu$  is a simple root 
such that $\langle -\mu , \gamma \rangle =-1$. This is a contradiction to the assumption  that 
$-\mu+2\gamma=-\mu_{1}\in R$ ( For instance, since both $-\mu$ and $\gamma$ are simple 
roots, $s_{\gamma}(-\mu_{1})=-\mu-\gamma\notin R$ ).  

This completes the proof of {\it Step 1}.

Now, let $\phi=ws_{\gamma}$. By Corollary 5.6, we see that 
$H^{i}(\phi , H^{1}(s_{\gamma}, V))=0$ for every $i\geq 1$.
On the other hand, using  Lemma 2.6, we see that $H^{0}(s_{\gamma} , V)$ is a 
$B$-submodule of $V$. Hence, by using induction on $l(w)$, we see that 
$H^{i}(\phi, H^{0}(s_{\gamma}, V))=0$ for every $i\geq 2$. 

Now, the proof of the lemma follows by applying the above observations in the 
following long exact sequence of $B$-modules ( see Lemma 6.1 ):
\begin{align*}
&H^{1}(\phi , H^{0}(s_{\gamma}, V))\longrightarrow 
 H^{1}(w , V)\longrightarrow H^{0}(\phi, H^{1}(s_{\gamma}, V)) \longrightarrow  \cdots
 \\ 
& \cdots 
\longrightarrow H^{i}(\phi, H^{0}(s_{\gamma}, V)) \longrightarrow 
H^{i}(w, V) \longrightarrow 
H^{i-1}(\phi, H^{1}(s_{\gamma} , V)) \longrightarrow  \cdots
\end{align*}

\end{proof}

The following is a consequence of Lemma 6.2.

\begin{lemma}
Let $w\in W$. Let $V_{2}$ be a $B$-submodule of $\mathfrak{g}$ and 
$V_{1}$ be a $B$-submodule of $\mathfrak{g}$  containing $V_{2}$.  Then we have 
$H^{i}(w, V_{1}/V_{2})=0$ for all $i\geq 2$. 
\end{lemma}
\begin{proof} The proof is similar to that of Lemma 3.5(1).  We use Lemma 6.2 instead of Lemma 3.4.
\end{proof}
\begin{corollary}
Let $w\in W$ and  let $\alpha\in R$. 
Then we have 
$H^{i}(w, \alpha)=0$ for every $i\geq 2$.
\end{corollary}
\begin{proof}
Let $V_{1}:=\oplus_{\mu\leq \alpha}\mathfrak{g}_{\mu}$ denote the 
direct sum of the weight spaces of $\mathfrak{g}$ of weights $\mu$ satisfying
$\mu\leq \alpha$. Let $V_{2}:=\oplus_{\mu < \alpha}\mathfrak{g}_{\mu}$ 
denote the direct sum of the weight spaces of $\mathfrak{g}$ of weights $\mu$ 
satisfying $\mu < \alpha$. It is clear that $V_{2}$ is a $B$-submodule of $\mathfrak{g}$ and $V_{1}$ is a $B$-submodule of $\mathfrak{g}$ containing $V_{2}$.

Since the $B$-module $V_{1}/V_{2}$ is isomorphic to $\mathbb{C}_{\alpha}$, we have $H^{i}(w, \alpha)=H^{i}(w, V_{1}/V_{2})$ for every $i\geq 2$. Now, the proof of the corollary follows from
 Lemma 6.3. \end{proof}

The following theorem is a main result in the non simply laced case.

\begin{theorem}
Let $w \in W$.  Then we have
\begin{myenumerate}
\item
$H^{i}(X(w), T_{G/B})=0$ for every 
$i\geq 1$. 
\item
The adjoint representation $\mathfrak{g}$ of $G$ is a $B$-submodule of 
$H^{0}(X(w) , T_{G/B})$  if and 
only if $w^{-1}(\alpha_{0})$ is a negative root.
\end{myenumerate}
\end{theorem}
\begin{proof}
The proof is similar to that of Theorem 4.1.
We provide a proof here for completeness.

As in the proof of Theorem 4.1, it is sufficient to prove the following:
\begin{myenumerate}
\item
$H^{i}(w, \mathfrak{g}/\mathfrak{b})=0$ for every $i\geq 1$. 
\item
The adjoint representation $\mathfrak{g}$ of $G$ is a $B$-submodule of 
$H^{0}(w , \mathfrak{g}/\mathfrak{b})$ if and only if $w^{-1}(\alpha_{0})$ 
is a negative root. 
\end{myenumerate}

{\it Proof of (1):} As in the proof of Theorem 4.1, let $V_{1}:=\mathfrak{g}$ and  
$V_{2}:=\mathfrak{b}$. The natural projection $\Pi:\mathfrak{g} \longrightarrow 
\mathfrak{g}/\mathfrak{b}$ of $B$-modules induces a 
homomorphism $\Pi_{w}:H^{0}(w, \mathfrak{g})
\longrightarrow H^{0}(w, \mathfrak{g}/\mathfrak{b})$ of $B$-modules.

We have the short exact sequence    
 $0\longrightarrow \mathfrak{b} \longrightarrow \mathfrak{g} \longrightarrow 
\mathfrak{g}/\mathfrak{b} \longrightarrow 0$ of $B$-modules.

Applying $H^{0}(w, -)$ to this short exact sequence of $B$-modules,  we obtain 
the following long exact sequence of $B$-modules: 

$$\cdots H^{i}(w, \mathfrak{b}) \longrightarrow H^{i}(w, \mathfrak{g}) 
\longrightarrow H^{i}(w, \mathfrak{g}/\mathfrak{b})\longrightarrow 
H^{i+1}(w, \mathfrak{b}) \cdots $$
On the other hand, by Lemma  2.5(2), we have 
$H^{i}(w, \mathfrak{g})=0$ for every $i\geq 1$. Further, by Lemma 6.2,  
we have $H^{i+1}(w, \mathfrak{b})=0$ for every $i\geq 1$.
Applying this in the above long exact sequence of $B$-modules, we conclude 
that $H^{i}(w, \mathfrak{g}/\mathfrak{b})=0$ for every $i\geq 1$. 

This proves (1).

The proof of (2) is similar to that of Theorem 4.1.
\end{proof}

The following result is the analogue of Theorem 4.2  in the non simply laced case.

\begin{theorem} 
Let $w\in W$ be such that $X(w)$ is smooth. Let $A_{w}$ denote 
the connected component of the group of all automorphisms of $X(w)$ containing the identity automorphism. For the left action of $G$ on $G/B$, let  $P_{w}$ denote the stabiliser of $X(w)$ in 
$G.$ The homomorphism  $\phi_{w}:P_{w} \longrightarrow A_{w}$ of algebraic groups induced by the action of $P_{w}$ on $X(w)$ is injective if and only if  
$w^{-1}(\alpha_{0})$ is a negative root.
\end{theorem}
\begin{proof}
The proof is similar to that of Theorem 4.2(2).
\end{proof}

The following remark illustrates the difference between the main results in the simply laced 
case and those in the non simply laced case.

\begin{remark} 
The adjoint representation $\mathfrak{g}$ of $G$ could be a proper $B$-submodule of  the space 
$H^{0}(X(w) , T_{G/B})$ of global sections. Let $G$ be of type $B_{2},$ and $w=s_{2}s_{1}s_{2}.$ Here, we follow the convention with $\langle \alpha_{1} , \alpha_{2} \rangle =-2$ and $\langle \alpha_{2} , \alpha_{1} \rangle =-1.$  In this case,  $H^{0}(X(w) , T_{G/B})$ is a direct sum of the $B$-modules $\mathfrak{g}$ and $\mathbb{C}_{-(\alpha_{1}+\alpha_{2})}$.
\end{remark}

{\bf Acknowledgements} 
We are very grateful to Professor M. Brion for pointing out this problem. 
We would like to thank the referee for useful comments to improve the exposition.
We would like to thank Professor C. S. Seshadri, A. J. Parameswaran and D. S. Nagaraj  for useful
discussions.


\begin{thebibliography}{22}
\bibitem[1]{r1}H.H. Andersen, Schubert varieties and Demazure's character formula, Invent. Math. {\bf 79} (1985), no. 3, 611-618.
\bibitem[2]{r2} V. Balaji, S. Senthamarai Kannan, K.V. Subrahmanyam, Cohomology of line bundles on Schubert varieties-I, Transformation Groups {\bf 9} (2004), no.2,  105-131.
\bibitem[3]{r3} R. Bott, Homogeneous vector bundles, Annals of Math., (2) {\bf 66} (1957), 203-248.
\bibitem[4]{r4} M. Brion, S. Kumar, Frobenius splitting methods in geometry and
 representation theory, Progress in Mathematics, Vol.{\bf 231}, Birkh\"auser, Boston, 
Inc., Boston, MA, 2005.
\bibitem[5]{r5} B. Narasmha Chary, S. Senthamarai Kannan, A.J. Parameswaran, 
Automorphism group of a Bott-Samelson-Demazure-Hansen variety,  Transformation Groups {\bf 20} (2015), no.3, 665-698. 
\bibitem[6]{r6} M. Demazure, A very simple proof of Bott's theorem, Invent. 
Math.  {\bf 33} (1976), 271-272.
\bibitem[7]{r7}A. Grothendieck, Techniques de construction et th\'{e}or\`{e}mes d'existence en 
g\'{e}om\'{e}trie alg\'{e}brique IV , les sch\'{e}mas de Hilbert, S\'{e}minaire Bourbaki {\bf 5 }(1960-1961), Expos\'{e} no. 221, 28 pp.
\bibitem[8]{r8} R. Hartshorne, Algebraic Geometry , Graduate Texts in 
Mathematics, Springer Verlag, Berlin Heidelberg, 1977.
\bibitem[9]{r9} J.E. Humphreys, Introduction to Lie algebras and 
Representation theory, Springer-Verlag, Berlin Heidelberg, New York, 1972.  
\bibitem[10]{r10} J.E. Humphreys, Linear Algebraic Groups, Springer-Verlag, 
Berlin Heidelberg, New York, 1975.  
\bibitem[11]{r11} J.C. Jantzen, Representations of Algebraic Groups, ( Second Edition ), Mathematical Surveys and Monographs, Vol. {\bf 107}, 2003.   
\bibitem[12]{r12} S. Senthamarai Kannan, Cohomology of line bundles on Schubert 
varieties in the Kac-Moody setting, J. Algebra {\bf 310} (2007),  no. 1, 88-107.
\bibitem[13]{r13} S.S. Kannan, S.K. Pattanayak, Torus quotients of homogeneous 
spaces-minimal dimensional Schubert varieties admitting semi-stable points, 
Proc. Indian Acad. Sci.(Math. Sci.)  {\bf 119} (2009), no.4,  469-485.
\bibitem[14]{r14}P. Littelmann, Contracting modules and standard monomial theory for symmetrizable Kac-Moody algebras. J. Amer. Math. Soc. {\bf 11} (1998), no. 3, 551-567.
\bibitem[15]{r15} H. Matsumura, F. Oort, Representability of group functors, and automorphisms of algebraic schemes, Invent. Math. {\bf 4} (1967), 1-25.  
\bibitem[16]{r16} V. B. Mehta, A. Ramanathan,  Frobenius splitting and cohomology vanishing for Schubert varieties. Ann. of Math. (2)  {\bf 122} (1985), no. 1, 27-40. 
\bibitem[17]{r17} D. Mumford, J. Fogarty and F. Kirwan, Geometric Invariant 
Theory, (Third Edition), Springer-Verlag, Berlin Heidelberg, New York, 1994. 
\bibitem[18]{r18}  C. S. Seshadri, Introduction to the theory of Standard Monomials. With notes by Peter Littelmann and Pradeep Shukla. Appendix A by V. Lakshmibai. Revised reprint of lectures published in the Brandeis Lecture Notes series. Texts and Readings in Mathematics, 46. Hindustan Book Agency, New Delhi, 2007. 
\bibitem[19]{r19} C. Weibel, An introduction to homological algebra, Cambridge 
University Press, 1994.  
\end{thebibliography}
\end{document}